\documentclass[aop,MSNbibl,nameyear,seceqn,dvips]{arximspdf}

\doi{10.1214/10-AOP623}
\volume{40}
\issue{1}
\pubyear{2012}
\firstpage{162}
\lastpage{212}

\makeatletter
\newcommand{\onrarrow}[1]{\mathrel{\vbox{\m@th\ialign{##\crcr
  $\hfil\scriptstyle\ #1\ \ \hfil$\crcr\noalign{\kern0.5pt\nointerlineskip}%
  \rightarrowfill\crcr}}}}

\newcommand{\mc}{\mathcal}
\newcommand{\izf}{\int_{0}^{\infty}}
\newcommand{\indic}{\mathbf{1}}
\newcommand{\tx}{\tilde{x}}
\newcommand{\on}{\operatorname}
\newcommand{\lkap}{\lambda_{0}^{\kappa}}
\newcommand{\supp}{\operatorname{supp}}
\newcommand{\tp}{\tau_\partial}
\renewcommand{\P}{\mathbb{P}}
\newcommand{\Rho}{\Gamma}
\newcommand{\R}{\mathbb{R}}
\newcommand{\rrho}{\gamma}
\newcommand{\mL}{\mathfrak{L}}
\newcommand{\cond}{|}
\newcommand{\mD}{\mathcal{D}}
\newcommand{\E}{\mathbb{E}}
\renewcommand{\P}{\mathbb{P}}
\newcommand{\tlam}{\tilde\lambda}
\newcommand{\tg}{\tilde{g}}
\newcommand{\eqref}[1]{(\ref{#1})}

\newtheorem{theorem}{Theorem}[section]
\newtheorem{Corollary}[theorem]{Corollary}
\newtheorem{lemma}[theorem]{Lemma}
\newtheorem{Prop}[theorem]{Proposition}

\newproclaim{definition}[theorem]{Definition}
\newproclaim{remark}[theorem]{Remark}
\newtheorem{lemmaA}{Lemma}

\makeatother

\begin{document}
\begin{frontmatter}

\title{Quasilimiting behavior for one-dimensional diffusions with killing\thanksref{TT1}}
\runtitle{Quasilimiting behavior}

\begin{aug}
\author[A]{\fnms{Martin} \snm{Kolb}\ead[label=e1]{kolb@stats.ox.ac.uk}}
and
\author[A]{\fnms{David} \snm{Steinsaltz}\corref{}\ead[label=e2]{steinsal@stats.ox.ac.uk}}

\runauthor{M. Kolb and D. Steinsaltz}

\affiliation{University of Oxford}

\address[A]{Department of Statistics\\
University of Oxford\\
1 South Parks Street\\
Oxford OX1 3TG\\
United Kingdom\\
\printead{e1}\\
\phantom{E-mail:\ }\printead*{e2}}
\end{aug}
\thankstext{TT1}{Supported in part by the New Dynamics of Ageing program.}

% HISTORY:
\received{\smonth{4} \syear{2010}}
\revised{\smonth{8} \syear{2010}}

% ABSTRACT
%
\begin{abstract}
This paper extends and clarifies results of Steinsaltz and Evans
[\textit{Trans. Amer. Math. Soc.} \textbf{359} (\citeyear{quasistat}) 1285--1234], which found
conditions for convergence of a killed
one-dimensional diffusion conditioned on survival, to a quasistationary
distribution whose density is given by the principal eigenfunction of
the generator. Under the assumption that the limit of the killing at
infinity differs from the principal eigenvalue we prove that
convergence to quasistationarity occurs if and only if the principal
eigenfunction is integrable. When the killing at $\infty$ is
\textit{larger} than the principal eigenvalue, then the eigenfunction
is always
integrable. When the killing at $\infty$ is \textit{smaller}, the
eigenfunction is integrable only when the unkilled process is
recurrent; otherwise, the process conditioned on survival converges to
0 density on any bounded interval.
\end{abstract}

% KEYWORDS
%
\begin{keyword}[class=AMS]
\kwd[Primary ]{60J60}
\kwd{60J70}
\kwd[; secondary ]{60J35}
\kwd{47E05}
\kwd{47F05}.
\end{keyword}

\begin{keyword}
\kwd{Killed one-dimensional diffusions}
\kwd{quasi-limiting distributions}.
\end{keyword}

\end{frontmatter}

%s1 ###
%s1 #&#
\section{Introduction} \label{sec:intro}

%s1.1 ###
%s1.1 #&#
\subsection{Background and history} \label{sec:background}

Killed Markov processes are central objects in probability theory. One
natural line of inquiry runs to questions about the asymptotic behavior
of the process conditioned on long-term survival.

We work with the one-dimensional diffusions $(X_t)_{t \geq0}$ on the
interval $[0,\infty)$, generated by the differential expression
$-L:=\frac{1}{2}\frac{d^2}{dx^2}+b\frac{d}{dx}$. In addition to the\vspace*{1pt}
possible killing at the boundary 0, there is is a killing rate $\kappa
$, so that we will really be concerned with the differential expression
$-L^{\kappa}:=-L-\kappa$. (We leave the description of the domain,
and hence of the operator and attendant semigroup, for later, because
much of the analysis will depend on moving flexibly among various
domains on which this differential expression can operate.) Let $\nu$
be a compactly supported distribution on $[0,\infty)$. We aim to find
conditions which imply convergence of the family of distributions
%
%e1.1 ###
%e1.1 #&#
\begin{equation}\label{mu}
\mu^{\nu}_t(\cdot):=\mathbb{P}_{\nu} (X_t \in\cdot\mid\tp
> t )
\end{equation}
as $t \rightarrow\infty$. This limit is sometimes called the
\textit{Yaglom limit}, after the seminal work of \citet{aY47} on
branching Markov processes conditioned on long survival. Any such limit
must be \textit{quasistationary}, in the sense that when started in
this distribution the process will remain in a multiple of the same
distribution for all times. The extensive mathematical development and
wide-ranging applications in this area---a bibliography of papers on
quasistationary distributions and Yaglom limits compiled and
periodically updated by \citet{quasibiblio} lists 403 entries
through 2010---permit us to mention only a smattering of the vast array
of applications of killed Markov processes to biology
[\citet{SV66}, \citet{gH97}, \citet{HT05},
\citet{6authors}, demography: \citet{diffmort},
\citet{hlB76}, \citet{LA09}, medicine: \citet{MS88}, \citet{aY07}
and statistics: \citet{oA95}, \citet{AG03}, \citet{dM04}]. Particularly in the
demographic and medical contexts, where killed Markov processes suggest
themselves as models for populations undergoing culling by mortality or
other processes, Yaglom limits, while rarely mentioned explicitly in
the applied literature, correspond naturally to the observable
distribution of survivors.

The central concerns of this theory are to describe, for a given class
of sub-Markov processes, the quasistationary distributions (if any),
and to describe the convergence (or not) of the process conditioned on
survival to one of these quasistationary distributions. A significant
part of the literature focuses on discrete state spaces, commonly
birth--death processes (or with some more flexible localization of the
transitions), with killing only on the boundary. One of the most
general accounts of the existence and convergence to quasistationary
distributions for discrete processes of this kind can be found in
\citet{FKMP95}. One unusual contribution, outside of these
categories, is \citet{fG01}, which proves convergence to
quasistationarity for fairly arbitrary discrete Markov chains with
general killing, by imposing a stringent Lyapunov-like drift condition.
Existence and vague-convergence conditions for discrete-time Markov
chains on general metric spaces can be found in \citet{LP01}.

Killed birth--death processes naturally generalize to killed diffusions
in the continuous-space context, but these have received rather less
attention. The existence of eigenfunctions for the generator of a
one-dimensional diffusions is simplified in the continuous setting, as
we may rely upon standard theory of ordinary differential equations.
Showing that these eigenfunctions are integrable (hence represent the
densities of distributions), and quasistationary is more involved,
though, and showing Yaglom convergence to the minimal quasistationary
distribution becomes technically challenging, particularly when the
state space is an unbounded interval. The foundation for all later work
on Yaglom convergence of diffusions was laid by \citet{pM61},
who used standard results from Sturm--Liouville theory and the spectral
theorem for self-adjoint operators to prove vague convergence (i.e.,
convergence of the distribution of the process conditioned on being in
a compact set), and uniform convergence under an assumption of strong
inward drift. These results have been substantially extended by a
shifting coalition of researchers who have produced papers
\citet{CMSM95}, \citet{MSM01}, \citet{6authors}, which
elucidate the conditions under which Yaglom convergence occurs, and
distinguish in \citet{MSM04} the $R$-positive situation from the
$R$-null---essentially, exponential-rate decay of probabilities
distinguished from decays that are asymptotically not exactly
exponential---in terms of the eigenfunctions.

One important constraint in most work in this field to date---as well
as Pinsky's results in \citet{rP85}, for diffusions on a compact
domain with gradient-type drift---has been the assumption that killing
occurs only at the boundary. Not only is this restriction unnatural
from the perspective of many of the applications, particularly the
demographic applications discussed in \citeauthor{quasistat} (\citeyear{diffmort}, \citeyear{quasistat}), it
obscures the fundamental links among the spectrum, the killing rate out
at infinity, the recurrence-transience dichotomy and Yaglom
convergence. (An exception which proves the rule is the biological
application of internally killed diffusions [\citet{KT83}], which makes
no reference to any of the literature on killed diffusions and cites
only \citet{eS66} for quasistationary distributions of discrete chains.)

Note that this approach to conditioning is quite different from Doob's
\textit{h-process} or \textit{h-transform}. We can generate an $h$-transform of
a Markov process which corresponds to conditioning on the process
\textit{never} being killed. That is, we look at the distribution of $\{
X_{t}\dvtx t\in[0,s]\}$ for fixed $s$ conditioned on $\tp>T$ (where $\tp$
is the killing time), in the limit as $T\to\infty$; we may then take
a second limit $s\to\infty$ to define the process on $[0,\infty)$.
This procedure generally produces a new Markov process, which is now
unkilled. A well-known example of this is the three-dimensional Bessel
process, which may be derived from the one-dimensional Brownian motion,
conditioned never to hit 0 [\citet{sV07}, Section 6.6]. This is
intimately connected to questions about the Martin boundary.

We will be concerned here only with the Yaglom approach, conditioning
on survival up to finite times. A key difference is that collection of
distributions~$\mu^{\nu}_t$ for different times $t$ are not
consistent, and so cannot be analyzed directly with Markov-process
techniques. They are more amenable to an analytic semigroup
approach.

%s1.2 ###
%s1.2 #&#
\subsection{Heuristics} \label{sec:heur}
We begin by observing that general spectral theory---sum\-marized here in
Lemma \ref{L:sqrtcompare}---tells us that the bottom of the
spectrum~$\lkap$ gives the exponential rate of decay of the distribution of
$X_t$ restricted to a compact interval.
What needs to be addressed, then, is the question of whether the
portion of the surviving mass within a compact interval dominates the
total surviving mass. There are two ways of addressing this question.
One is in terms of the spectrum of the $\mL^{2}$ generator. Suppose
$K:=\lim_{x\to\infty}\kappa(x)$ exists. In \citet{quasistat} the
emphasis was placed on the crucial distinction between the cases $\lkap
>K$ and $\lkap<K$.

It turns out that a more useful dichotomy is whether or not $\lkap$ is
an isolated eigenvalue. The eigenvalue $\lkap$ is isolated when the
diffusion takes place on a compact interval with two regular
boundaries, but also in cases which intuitively seem well-approximated
by a compact process, as when there is strong drift pulling the process
in from $\infty$, and when there is strong killing out toward $\infty
$. Thus the isolated-eigenvalue case includes all of the $\lkap<K$
case [see Lemma \ref{spectrum}\hyperlink{it:isolated}{(v)}]. We expect the same
methods that work in finite dimensions, for powers of positive
symmetric matrices, to work in this case as well. One catch is that the
$\mL^{2}$ convergence need not tell us about the convergence of the
conditioned density, which is an $\mL^{1}$ property; indeed, it is
easy to see [cf. Proposition 2.3 of \citet{quasistat}] that the process
always escapes to $\infty$ if $\izf\varphi(\lkap,x)\,d\Rho(x)$ is
not finite, where~$\Rho$ is the speed measure and $\varphi(\lambda
,\cdot)$ is the eigenfunction of the generator, defined as the
solution to an ordinary differential equation in Section \ref
{sec:SLspec}. We show, in Lemma \ref{L:L1L2}, that these conditions
do, in fact, suffice: That is, whenever~$\lkap$ is an isolated
eigenvalue, and the corresponding eigenfunction is also integrable,
then we have convergence to the quasistationary distribution given by
the density $\varphi(\lkap,\cdot)/\int\varphi(\lkap,x)\,d\Rho(x)$.

What about the case when $\lkap$ is not an isolated eigenvalue? This
corresponds to the $R$-null and $R$-transient cases in Tweedie's theory
[\citeauthor{rT74} (\citeyear{rT74}, \citeyear{rT742})], where the decay of the transition kernel is not
exactly exponential with rate $-\lkap$, but slightly faster, in the
sense that $e^{\lkap t}p^{\kappa}(t,x,y)\to0$ ($p^{\kappa}$ being
the diffusion transition kernel). It turns out that in this case the
convergence lines up precisely with the standard recurrence/transience
dichotomy for the unkilled process. [Another way of putting this is to
say that when the $R$-recurrence or $R$-transience does not conform to
the properties of the unkilled process, this must be reflected in the
equality of $\lkap$ and $\lim_{x\to\infty} \kappa(x)$.]

Another way of understanding the nonisolated case is by thinking about
how the condition $\lkap>K$ implies that the distribution must decline
on compact sets at a faster exponential rate than would keep pace with
the killing out toward $\infty$. There are two ways this imbalance in
killing can be maintained: Either the mass vanishes toward $\infty$,
meaning that the scale (of the unkilled diffusion) is finite; this is
the $R$-transient case. Or the scale is infinite with finite speed,
which means that the excess mass keeps returning to 0, at long
intervals, and the killing rate $\lambda_{0}$ corresponds to real
killing at~0, not escape; this is the $R$-null case. In the
$R$-transient case the conditioned process escapes to infinity. In the
$R$-null case the conditioned process converges to the quasistationary
distribution. Note that the arguments for the one or the other behavior
seem to refer only to the motion, irrespective of the killing $\kappa$.

%s1.3 ###
%s1.3 #&#
\subsection{Main results} \label{sec:main}
The core of this work is the identification of the asymptotic behavior
of killed diffusions in terms of the relation between the principal
eigenvalue of the generator, the limit behavior of $\kappa$ and the
nature of the boundary at~$\infty$. We move beyond earlier work in
removing unnecessary constraints on the drift and killing terms, and in
providing easily testable criteria for determining whether the
conditioned process converges to a quasistationary distribution for all
cases in which the bottom of the spectrum~$\lkap$ does not coincide
with the limit of the killing rate at $\infty$.

We begin by summarizing the most important results. These results
presuppose general assumptions and restrictions on the processes
involved, which will be formulated fully in Section
\ref{sec:ADPR}.
The quasistationary distribution will be defined in terms of its
density $\varphi(\lkap,\cdot)$ with respect to $\Rho$, where
$\varphi(\lambda_0^{\kappa},\cdot)$ is the principal eigenfunction
of the generator, defined as the solution to an ordinary differential
equation with appropriate boundary condition, stated formally in
Section \ref{sec:SLspec}.

\begin{enumerate}[(iii)]
\item[(i)] \textit{Convergence on compacta}: There is always convergence
to the quasistationary distribution on compact sets, stated formally as
Theorem~\ref{generallocalMandl}.\hypertarget{it:generallocalMandl}{}%
\item[(ii)] \textit{Dichotomy}: If $\lkap>\limsup_{x\to\infty} \kappa
(x)$ or $\lkap<\liminf_{x\to\infty} \kappa(x)$, then the
conditioned process either converges to the quasistationary
distribution with density $\varphi(\lkap,\cdot) (\izf\varphi
(\lkap,y)\, d\Rho(y) )^{-1}$ with respect to $\Rho$, or escapes
to $\infty$. This behavior is independent of the initial distribution,
provided only that it is compactly supported. (For explanation of the
terminology, see Section \ref{sec:quasi}.) This is Theorem 3.3 of
\citet{quasistat}, but it is restated here as Theorem \ref{Thm33} in
a slightly stronger form, as several restrictions have been
removed.\hypertarget{it:compacta}{}
\item[(iii)] \textit{Yaglom convergence with high killing at $\infty$
always}: If $\lkap<\liminf_{x\to\infty}\kappa(x)$, then the
conditioned process converges to the quasistationary distribution. This
is stated as Theorem \ref{limkbiggerev}. Note that this includes the
(somewhat unintuitive) fact that a bound on the $\mL^{2}$ spectrum
implies that $\izf\varphi(\lkap,y)\,d\Rho(y)<\infty$, which is a
fact about the $\mL^{1}$ spectrum.\hypertarget{it:bigK}{}
\item[(iv)] \textit{Yaglom convergence with low killing at $\infty$ when
recurrent}: If $K:=\lim_{x\to\infty} \kappa(x)$ exists and $K<\lkap
$, then the behavior of the conditioned process depends on the
transience or recurrence of the unkilled process. If the unkilled
process is transient---that is, if $\izf\rrho(x)^{-1}\,dx<\infty
$---then the conditioned process escapes to $\infty$. If the unkilled
process is recurrent---that is, if $\izf\rrho(x)^{-1}\,dx=\infty
$---then the conditioned process converges to the quasistationary
distribution. These results are stated in Theorems \ref{Escape} and
\ref{qsddrift}.\hypertarget{it:smallK}{}
\item[(v)] \textit{Yaglom convergence equivalent to integrability of the
principal eigenfunction}: If $\lkap< \liminf_{x \rightarrow\infty
}\kappa(x)$ or $\lkap> \lim_{x \rightarrow\infty}\kappa(x)$, then
convergence to quasistationarity is equivalent to the integrability of
the principal eigenfunction $\varphi(\lambda_0^{\kappa},\cdot)$.
This is stated as Theorem \ref{thm:integrability} \hypertarget{it:integrab}{}
\end{enumerate}
Our results extend those of \citet{quasistat} in several ways:
\begin{itemize}
\item In \citet{quasistat} the authors had to impose conditions that
required the drift and killing not to grow too quickly, or be too
irregular in order to insure that $\infty$ is of the limit point type.
Here there is no constraint on the killing other than local
boundedness, and no constraint on the drift other than that which
implies that $\infty$ is inaccessible. The case of an entrance
boundary at infinity was excluded in \citet{quasistat}. Moreover, in
contrast to \citet{quasistat} item \hyperlink{it:generallocalMandl}{(i)} is
shown to hold without any further condition on the initial distribution
other than compact support.
\item In \citet{quasistat} an assertion of the type \hyperlink{it:bigK}{(iii)}
was shown under the assumption that $K = \lim_{x \rightarrow\infty
}\kappa(x)$ exists and some further growth restrictions on $b$ and
$\kappa$.
\item Item \hyperlink{it:smallK}{(iv)} describes the most substantial advance:
The case $\lim_{x \rightarrow\infty}\kappa(x)<\lkap$ is now shown
to be split by the standard recurrence-transience dichotomy, which
tells us whether the conditioned process converges or escapes. In
\citet{quasistat} the dichotomy could not be decided if $\lim_{x
\rightarrow\infty}\kappa(x)<\lkap$.
\item In \citet{quasistat} the assertion of item \hyperlink{it:integrab}{(v)} was established only in the case $\lkap< \liminf_{x
\rightarrow\infty}\kappa(x)$.
\end{itemize}

%s2 ###
%s2 #&#
\section{Assumptions, definitions and previous results} \label{sec:ADPR}
%s2.1 ###
%s2.1 #&#
\subsection{Analytic terminology} \label{sec:analytic}
In general a Sturm--Liouville operator is any formal differential
operator of the form $\tau= \tau_{p,q,V}= -\frac{1}{2p}\frac
{d}{dx}q\frac{d}{dx} + V$, where $p,q\dvtx (a_1,a_2) \rightarrow(0,\infty
)$ and $V\dvtx(a_1,a_2)\rightarrow\mathbb{R}$ are sufficiently
well-behaved functions. In this work we consider only operators where
$p=q=\rrho$, $V=\kappa\geq0$ and $a_1=0$, $a_2=\infty$. Note that
the diffusion coefficient has been set to~$1$. However, the case of a
general nondegenerate diffusion coefficient can be reduced to the
present case via a time change. Thus our results can be applied to the
case of a general diffusion coefficient. This reduction simplifies the
formulas considerably.
Moreover, we always assume in this chapter that $\rrho(x) = e^{2\int
_0^x b (s)\,ds}$ for some $b \in\mL^{1}_{loc}([0,\infty))\cap
C((0,\infty))$ and $0 \leq\kappa\in C ([0,\infty) )$.
These conditions are not entirely necessary, but this constraint still
admits a large class of one-dimensional diffusions. [However, see \citet{6authors} for a natural application to biology which requires $b$ to
be singular at 0.] Concerning the assumptions on $b$ we could replace
the condition $b \in\mL^{1}_{loc}([0,\infty))$ by the condition that
$\int_0^1 e^{-\int_c^x2b(s)\,ds}\,dx<\infty$ for some $c\in
(0,\infty)$, and $\int_0^1 e^{\int_c^x 2b(s)\,ds}\,dx<\infty$,
which is equivalent to saying that the boundary point $0$ is regular in
the sense of Feller and also in the sense of Weyl. In this paper we
will consistently use $\Rho$ as a reference measure instead of the
Lebesgue measure, which is different from the convention adopted in
\citet{quasistat}. Recall that the speed measure of a one-dimensional
diffusion is also the reversing measure, with respect to which the
generator is symmetric. Unless otherwise indicated, we will always use
the bare notation $\mL^{2}$ to mean $\mL^{2} ((0,\infty),\Rho
 )$, and for $f,g\in\mL^{2}$ we have the inner product
%
%e2.1 ###
%e2.1 #&#
\begin{equation} \label{E:innerprod}
\langle f,g\rangle=\izf f(x)g(x)\rrho(x) \,dx.
\end{equation}
The formal differential operator $L^{\kappa} = -\frac{1}{2\rrho
}\frac{d}{dx}\rrho\frac{d}{dx} + \kappa$ gives rise to a closable
densely defined quadratic form $\tilde{q}^{\kappa, \alpha}$ in $\mL
^2$ by
%
%e2.2 ###
%e2.2 #&#
\begin{eqnarray} \label{E:defineqk}
\qquad &&\varphi\mapsto\tilde{q}^{\kappa,\alpha}(\varphi)\nonumber\\[-8pt]\\[-8pt]
\qquad &&\qquad  =
\cases{
\displaystyle \alpha\varphi(0)^{2}+\frac{1}{2}\int_0^{\infty}|\varphi
'(y)|^2\rrho(y)\,dy + \int_0^{\infty}\kappa(y)|\varphi(y)|^2
\rrho(y)\,dy, \vspace*{2pt}\cr
\qquad   \mbox{if $\alpha<\infty$,}\vspace*{4pt}\cr
\displaystyle \frac{1}{2}\int_0^{\infty}|\varphi'(y)|^2\rrho(y)\,dy + \int
_0^{\infty}\kappa(y)|\varphi(y)|^2 \rrho(y)\,dy,\vspace*{2pt}\cr
\qquad  \mbox{if $\alpha=\infty$},
}\nonumber
\end{eqnarray}
for any $\varphi\in\mD_{\kappa,\alpha}$, where $\mD_{\kappa
,\alpha}$ is defined by
\[
\mathcal{D}_{\kappa, \alpha}:=
\cases{
 \lbrace\varphi\in\mL^2 |\varphi\in C^1(0,\infty)\cap
C([0,\infty)),  \tilde{q}^{\kappa,\alpha}(\varphi)<\infty
\rbrace,\vspace*{1pt}\cr
\qquad  \mbox{if $\alpha\in[0,\infty)$,} \vspace*{3pt}\cr
 \lbrace\varphi\in\mL^2 |\varphi\in C^1(0,\infty)\cap
C([0,\infty)), \varphi(0)=0,  \tilde{q}^{\kappa,\infty}(\varphi
)<\infty \rbrace,\vspace*{1pt}\cr
\qquad \mbox{if $\alpha= \infty$}.
}
\]
The closure of this quadratic form will be denoted by $q^{\kappa
,\alpha}$. To the quadratic form $q^{\kappa,\alpha}$ there
corresponds a uniquely defined positive self-adjoint 
operator~$L^{\kappa,\alpha}$ with a dense domain of definition $\mathcal
{D}(L^{\kappa,\alpha})$. It is easy to see (essentially via
integration by parts) that the action of the operator $L^{\kappa
,\alpha}$ is given by
\[
L^{\kappa,\alpha}\varphi(x) = -\tfrac{1}{2}\varphi''(x) - b(x)
\varphi'(x) + \kappa(x) \varphi(x).
\]
By definition of the operator $L^{\kappa,\alpha}$ every element
$\varphi\in\mathcal{D}(L^{\kappa,\alpha})$ is absolutely
continuous and satisfies the boundary condition $2\alpha\varphi(0)=
\varphi'(0)$ [or $\varphi(0)=0$ when $\alpha=\infty$]. As in the
definition of $q^{\kappa,\alpha}$ we see that $\alpha=\infty$
corresponds to Dirichlet condition at $0$ (instantaneous killing), and
$\alpha= 0$ to Neumann condition (pure reflection) at $0$.

The bottom of the spectrum of $L^{\kappa}$ will be denoted by $\lambda
^{\kappa}_0$. The spectrum of the self-adjoint operator $L^{\kappa}$
is written $\Sigma(L^{\kappa})$. Where there is no danger of
confusion, the corresponding objects with $\kappa\equiv0$ will also
be denoted by~$q$, $L$ and $\lambda_0$ instead of $q^0$, $L^0$ and
$\lambda_0^0$, respectively (or $q^{0,\alpha}$, $L^{0,\alpha}$ and
$\lambda_{0}^{0,\alpha}$). Since $L^{\kappa}$ and $L$ are
self-adjoint operators, the spectral theorem implies the existence of
spectral resolutions $(E^{\kappa}_{\lambda})_{\lambda\in[\lambda
_0^{\kappa},\infty)}$ and $(E_{\lambda})_{\lambda\in[\lambda
_0,\infty)}$, respectively. For the basic facts concerning spectral
theory of self-adjoint operators the reader should consult \citet{jW00}.

The spectral theorem for self-adjoint operators allows us to define
functions $f(L^{\kappa})$ of the operator. For every Borel-measurable
function $f\dvtx\mathbb{R} \rightarrow\mathbb{R}$ the operator
$f(L^{\kappa})$ is defined via
%
%e2.5 ###
%e2.4 ###
%e2.3 ###
%e2.3 #&#
%e2.4 #&#
%e2.5 #&#
\begin{eqnarray}
\mathcal{D}(f(L^{\kappa}))&=& \biggl\lbrace u \in\mL^2 \Big|\int
_{\Sigma(L^{\kappa})}|f(\lambda)| ^2\,d\|E^{\kappa}u\|^2(\lambda
)<\infty \biggr\rbrace,\label{spectraltheorem1}\\
f(L^{\kappa})u &=& \int_{\Sigma(L^{\kappa})}f(\lambda)\,dE^{\kappa
}(\lambda) u,\label{spectraltheorem2}\\
 \|f(L^{\kappa}) u \|^{2} &=& \int_{\Sigma(L^{\kappa})}
f(\lambda)^{2} \,d\|E^{\kappa}u\|^2(\lambda). \label{spectraltheorem3}
\end{eqnarray}
Observe that for a Borel-measurable function $f\dvtx[0,\infty)\rightarrow
\mathbb{R}$ and $a \geq0$ we have $\operatorname{Ran} (f(L^{\kappa})) \subset
\mathcal{D}((L^{\kappa})^{a})$ if $[0,\infty) \ni\lambda\mapsto
|\lambda^{a} f(\lambda)|$ is bounded. This implies in particular that
the range of $e^{-t L^{\kappa}}$ is contained in the domain of all
powers of $L^{\kappa}$. Moreover the spectral theorem allows us to
clarify further the connection between the quadratic form $q^{\kappa}$
and the associated nonnegative operator $L^{\kappa}$. Let $\sqrt
{L^{\kappa}}$ denote the unique nonnegative square root of $L^{\kappa
}$, which is defined using the spectral theorem. Then we have $\mathcal
{D}(q^{\kappa})=\mathcal{D}(\sqrt{L^{\kappa}})$, and for every $f
\in\mathcal{D}(L^{\kappa})$ we have
%
%e2.6 ###
%e2.6 #&#
\begin{equation}\label{formoperator}
q^{\kappa}(f,g) =  \bigl\langle\sqrt{L^{\kappa}}f, \sqrt{L^{\kappa
}}g \bigr\rangle.
\end{equation}
Using the ``elliptic'' Harnack inequality and Weyl's spectral theorem
it is not difficult to see that
%
%e2.7 #&#
\begin{eqnarray}\label{possolutions}
\lambda_0^{\kappa} &=& \max\biggl\lbrace\lambda\in\mathbb{R} \mid
\mbox{there is a positive solution of $(L^{\kappa}-\lambda)u=0$}
\nonumber\\[-8pt]\\[-8pt]
&&\hspace*{94pt}\mbox{with }u(0)=\frac{1}{1+\alpha},  \frac{1}{2}u'(0)=\frac
{\alpha}{1+\alpha} \biggr\rbrace.\nonumber
\end{eqnarray}
[This was proved by \citet{pM61} using slightly different methods.]
Equation \eqref{possolutions} already suggests that for $0 \leq
\lambda\leq\lambda_0^{\kappa}$ solutions of $(L^{\kappa}-\lambda
)u=0$ might have a probabilistic significance.

In the sequel we usually denote by $\varphi(\lambda,\cdot)$ the
solution of the eigenvalue equation
%
%e2.7 ###
%e2.8 #&#
\begin{equation}\label{eigenodequation}
(L^{\kappa}-\lambda)\varphi(\lambda,\cdot)=0,\qquad  \varphi
(\lambda,0)=\frac{1}{1+\alpha}, \frac{1}{2}\varphi'(\lambda
,0)=\frac{\alpha}{1+\alpha}.
\end{equation}
It might be important to note that solutions in \eqref{possolutions}
and \eqref{eigenodequation} are solutions in the sense of the theory
of ordinary differential equations. An important issue is whether the
solution also belongs to the Hilbert space $\mL^2$ and thus is an
eigenfunction in the sense of spectral theory. When we wish to
emphasize that certain solutions are also eigenfunctions in the sense
of spectral theory, we denote them by $u_{\lambda}$.

Crucial to much of our analysis is the fact that the asymptotic
behavior of the semigroup is wholly determined by the spectrum right
near the base of the spectral measure, which we show in Lemma \ref
{L:sqrtcompare}, and then that the base of the spectral measure for any
nonnegative function is $\lkap$, which is Lemma \ref{L:supportbase}.
For $g\in\mL^{2}$, define $\lambda_{g}$ to be the infimum of the
support of the spectral measure of $g$; that is,
%
%e2.8 ###
%e2.9 #&#
\begin{equation} \label{E:lambdag}
\lambda_{g}:= \sup \{\lambda \dvtx \|E _{\lambda} g \|=0 \},
\end{equation}
and let $\mc{A}_{\lambda}$ be the subspace of $\mL^{2}$ consisting
of functions $f$ such that $\lambda_{f}\ge\lambda$.

%le2.1 #&#
\begin{lemma} \label{L:sqrtcompare}
Given $g\in\mathcal{D}(L^{\kappa,\alpha})$, we have
%
%e2.10 ###
%e2.9 ###
%e2.10 #&#
\begin{eqnarray} \label{E:sqrtcompare}
 |g(x)|&\le& C_{\alpha}(x) \bigl\|\sqrt{L^{\kappa}} g \bigr\|
+C'_{\alpha} \| g \|\nonumber\hspace*{-35pt}\\[-8pt]\\[-8pt]
&=&C_{\alpha}(x)  \biggl(\int_{0}^{\infty}\lambda \,d\|E^{\kappa} g\|
^{2}(\lambda) \biggr)^{1/2}+C'_{\alpha}  \biggl(\int_{0}^{\infty}d\|
E^{\kappa} g\|^{2}(\lambda) \biggr)^{1/2},\nonumber\hspace*{-35pt}
\end{eqnarray}
where
%
%e2.11 ###
%e2.11 #&#
\begin{equation} \label{E:Calpha}
C_{\alpha}(x):=
\cases{
\displaystyle \max \biggl\{\sqrt{\frac{2}{\alpha}}, \biggl(2\int_{0}^{x}\rrho
(y)^{-1}\,dy \biggr)^{1/2}  \biggr\},&\quad for $\alpha>0$,\cr
\displaystyle  \biggl(18\int_{0}^{x}\rrho(y)^{-1}\,dy \biggr)^{1/2},&\quad  for $\alpha=0$,
}\hspace*{-35pt}
\end{equation}
and
%
%e2.12 ###
%e2.12 #&#
\begin{equation} \label{E:Cprimealpha}
\quad C'_{\alpha}:=
\cases{
0,&\quad if $\alpha>0$ or $ \displaystyle \int_{0}^{\infty
}\rrho(y)\,dy=\infty$,\cr
\displaystyle  \biggl(\int_{0}^{\infty} \rrho(y) \,dy  \biggr)^{-1/2},&\quad if
$\alpha=0$ and $\displaystyle \int_{0}^{\infty}\rrho(y)\,dy<\infty$.
}\hspace*{-35pt}
\end{equation}
For any $t>1/2\lambda_{g}$,
%
%e2.13 ###
%e2.13 #&#
\begin{equation} \label{E:sqrtcompare2}
\sup|e^{-tL^{\kappa}}g(x)|\le \bigl(C_{\alpha}(x)\lambda
_{g}+C'_{\alpha} \bigr)\|g\| e^{-t\lambda_{g}}.
\end{equation}
\end{lemma}

\begin{pf}
Suppose $\alpha\in(0,\infty)$. Since $g\in\mathcal{D}(L^{\kappa
})$ is differentiable, we have
\begin{eqnarray*}
|g(x)|&\le&|g(0)|+\izf|g'(y) |\frac{\indic_{[0,x]}}{\rrho(y)}
\rrho(y)\,dy\\
&=& |g(0)|+  \biggl\langle|g'|, \frac{\indic_{[0,x]}}{\rrho}
\biggr\rangle\\
&\le&|g(0)|+ \biggl\| \frac{\indic_{[0,x]}}{\rrho}  \biggr\| \cdot\|
g'\|\qquad   \mbox{(Cauchy--Schwarz inequality)}\\
&\le& \biggl(2 |g(0)|^{2}+ 2 \biggl\| \frac{\indic_{[0,x]}}{\rrho
} \biggr\| \izf|g'(y)|^{2}\rrho(y)\,dy  \biggr)^{1/2}\\
&\le& C_{\alpha} q^{\kappa,\alpha}(g)^{1/2}\\
&=& C_{\alpha}  \bigl\| \sqrt{L^{\kappa,\alpha}}g  \bigr\|
\end{eqnarray*}
by \eqref{E:defineqk} and \eqref{formoperator}.
The spectral theorem \eqref{spectraltheorem3} allows us to represent
$\sqrt{L^{\kappa}}g$ in terms of the spectral resolution, yielding
\eqref{E:sqrtcompare}.

If $\alpha=\infty$, then $g(0)=0$, so the corresponding term drops
out of the bound.

If $\alpha=0$, we have the alternative bound
\begin{eqnarray*}
|g(x)|&\le&|g(0)|+C  (q^{\kappa,0}(g) )^{1/2},\\
|g(x)|&\ge&|g(0)|-C (q^{\kappa,0}(g) )^{1/2},
\end{eqnarray*}
where $C=\sqrt{2} \| \frac{\indic_{[0,x]}}{\rrho}  \|$.
The second bound gives us
\[
\|g\|^{2}\ge \biggl( \int_{0}^{\infty}\rrho(y)\,dy  \biggr)  \bigl(
 | g(0) |^{2}-2C  | g(0) | (q^{\kappa
,0}(g) )^{1/2} \bigr),
\]
which implies that
\[
 | g(0) | \le2C\sqrt{q^{\kappa,0}(g)}+ \|g\|^{2}  \biggl(
\int_{0}^{\infty}\rrho(y)\,dy  \biggr)^{-1/2}.
\]
We combine this with the above calculation to obtain the appropriate
version of \eqref{E:sqrtcompare}.

For any positive $t$, we have $g_{t}:=e^{-tL^{\kappa}}g\in\mc
{D}(L^{\kappa})$, so we may apply \eqref{E:sqrtcompare} to obtain
\[
|g_{t}(x)|\le \biggl(\int_{0}^{x}\rrho(y)^{-1}\,dy \biggr)^{1/2}
\bigl\| \sqrt{L^{\kappa}}e^{-tL^{\kappa}}g  \bigr\|.
\]
Applying again the spectral theorem \eqref{spectraltheorem3}---now
with $f(x)=\sqrt{x} e^{-tx}$---yields
\begin{eqnarray*}
 \| \sqrt{L^{\kappa}}e^{-tL^{\kappa}}g  \|^{2}&=&\izf
\lambda e^{-2t\lambda} \,d \| E^{\kappa} g \|^{2} (\lambda
)\\
&=&\int_{\lambda_{g}}^{\infty} \lambda e^{-2t\lambda} \,d \|
E^{\kappa} g \|^{2} (\lambda)\\[-1pt]
%&\le (\max_{\lambda\ge\lambda_{g}} \lambda e^{-2t\lambda}  )
%g \|^{2} (\lambda)\\
&\le&\lambda_{g} e^{-2t\lambda_{g}}\|g\|^{2},
\end{eqnarray*}
since $\lambda e^{-2t\lambda}$ attains its maximum at $\lambda=\frac
{1}{2t}$. Similarly,
\[
 \|e^{-tL^{\kappa}}g  \|^{2}=\int_{\lambda_{g}}^{\infty}
e^{-2t\lambda} \,d \| E^{\kappa} g \|^{2} (\lambda)\le
e^{-2t\lambda_{g}}\|g\|^{2}.\vspace*{-2pt}
\]
\upqed
\end{pf}

%le2.2 #&#
\begin{lemma} \label{L:supportbase}
For any nonnegative measurable function $f \in\mL^2$ with\break \mbox{$\|f\|>0$},
the spectral measure $d\|E^{\kappa}f\|^{2}(\lambda)$ corresponding to
$f$ includes $\lkap$ in its support.\vspace*{-2pt}
\end{lemma}

\begin{pf}
Since $e^{-L^{\kappa}}f$ is everywhere nonnegative (except perhaps at
the boundary), and its associated spectral measure has the same support
as $d\|E^{\kappa}f\|^{2}$, we may assume that if there were a
counterexample it would not vanish off the boundary.

Suppose there is some $\lambda_{*}>\lambda_{0}^{\kappa}$ such that
$\|E^{\kappa}_{\lambda^{*}} f\|=0$.
Then for any $h \in\mL^2$ with $|h| \le f$,
\begin{eqnarray*}
e^{-\lambda_{*}t}\|f\|^{2}&\ge&\int_{\lambda_{*}}^{\infty}
e^{-\lambda t} \,d\|E^{\kappa}f\|^{2}(\lambda)\\[-1pt]
&=& \| e^{-({t}/{2})L^{\kappa}} f  \|\\[-1pt]
&\ge& \| e^{-({t}/{2})L^{\kappa}} h  \|\\[-1pt]
&=& \int_{\lkap}^{\infty} e^{-\lambda t} \,d\|E^{\kappa}h\|
^{2}(\lambda).
\end{eqnarray*}
Thus, it must be that $d\|E^{\kappa}h\|^{2}(\lambda)$ is supported on
$[\lambda_{*},\infty)$ as well. Thus, for all such $h$ we have $\|
E^{\kappa}_{\lambda^{*}} h\|=0$.

Let $f_{n}=f\cdot\indic_{[0,n]}$. For any $\tlam\in(\lkap,\lambda
_{*})$, by \eqref{spectraltheorem1} $f_{n}$ is in the domain of the
resolvent $R_{\tlam}=(L^{\kappa}-\tlam)^{-1}$. Furthermore, by
\eqref{spectraltheorem3}, if we choose $\lambda_{**}$ large enough so
that $\|\E_{\lambda_{**}} f\|>0$, then
\[
 \| R_{\tlam}f \|^{2} =\int(\tlam-\lambda)^{-2} \,d \|
E^{\kappa}f \|^{2}(\lambda)\ge(\tlam-\lambda_{**})^{-2} \|\E
_{\lambda_{**}} f\|^{2}>0.
\]
Let $g_{n}:=R_{\tlam} f_{n}/\|R_{\tlam}f_{n}\|$. Then $g_{n}$ satisfies
\[
L^{\kappa} g_{n}(x)=\tlam g_{n} \qquad \mbox{for }x\le n.
\]
(In principle, the equality holds only in the $\mL^{2}$ sense, but it
becomes true for all $x$ since both sides are in $\mD_{\alpha}$,
hence, in particular, continuous.) By the representation
\[
R_{\tlam}f_{n}=\int_{0}^{\infty} e^{s\tlam}e^{-sL^{\kappa}}f_{n} \,ds,
\]
we see that $g_{n}$ is nonnegative.\vadjust{\goodbreak}

The space of solutions to the ordinary differential equation $L^{\kappa
}g=\tlam g$ satisfying boundary condition \eqref{E:fellerbound} is
one-dimensional, so if we renormalize to
\[
\tg_{n}:=
\cases{
(1+\alpha)^{-1}g_{n}(0)^{-1}g_{n}, &\quad if $\alpha<\infty$,\cr
g'_{n}(0)^{-1}g_{n}, &\quad if $\alpha=\infty$,
}
\]
we have $\tg_{n}(x)=\tg_{n'}(x)$ for $x\in[0,n]$ when $n\le n'$. The
limit must then be identical with the function $\varphi(\tlam,\cdot
)$, and is everywhere nonnegative, contradicting the characterization
of $\lkap$ in \eqref{possolutions}.
%
%With one regular boundary and one limit-point boundary the spectral
%representation for the Sturm-Liouville operator is one-dimensional,
%given by the unitary map $U:\mL^{2}(\Rho)\to\mL^{2}(\sigma)$
%U g(\lambda)=\lim_{n\to\infty}\int_{0}^{n} \varphi(\lambda,y)g(y)
%where $\sigma$ is the spectral measure, and the limit is understood in
%the $\mL^{2}(\sigma)$ sense. Then we have
%$$
%$$
%implying that for $\sigma$-almost every $\lambda\in[\lambda_{0}^{
%$$
%$$
%Choose any $\lambda\in\on{supp}(\sigma)\cap[\lkap,\lambda_{*}]$, and
%set
%$$
%g_{\lambda}(y):=\sgn(\varphi(\lambda,y)) [\lv\varphi(\lambda,y)
%$$
%Then
%$$
%|f(y)| )\rrho(y) dy \convinfty{n}_{\mL^{2}(\sigma)} 0.
%$$
%Since the integral is an increasing function of $n$, this implies that
%$$
%$$
%for $\sigma$-a.e. $\lambda$. Thus $\varphi(\lambda,\cdot)$ vanishes
%Lebesgue-a.e. on the support of $f$, hence on a set of positive
%Lebesgue measure. This contradicts the fact that $\varphi(\lambda,
%fact, the Sturm Separation Theorem \cite[Theorem 2.6.1]{aZ05} tells us
%that the zeros are isolated.)
\end{pf}

%s2.2 ###
%s2.2 #&#
\subsection{Boundary conditions, recurrence and transience} \label{sec:BC}
Defining the diffusion includes a boundary condition at 0, parametrised
by $\alpha\in[0,\infty]$
%
%e2.14 ###
%e2.14 #&#
\begin{equation} \label{E:fellerbound}\quad
2\alpha\phi(0) = \phi'(0)\qquad  \mbox{if }\alpha<\infty,\quad \mbox{or}\quad \phi(0)=0 \qquad \mbox{if }\alpha=\infty.
\end{equation}
(It is more common in probability to use a parameter on $[0,1]$,
corresponding to $\alpha/(1+\alpha)$.) The condition $\alpha=\infty
$ corresponds to instantaneous killing at~0, while $\alpha=0$
corresponds to reflection with no killing. Intermediate parameters
correspond to ``slow killing'' at 0, so that the process is killed when
the local time at 0 reaches an exponentially distributed random
variable. The operator $L^{\kappa,\alpha}$ is associated with the
closure of the quadratic form~$\tilde{q}^{\kappa,\alpha}$.
That is, $L$ is the self-adjoint realization of the differential
expression $-\frac{1}{2\rrho}\frac{d}{dx} (\rrho\frac
{d}{dx} )+\kappa$ in $\mL^2$ that has boundary condition \eqref
{E:fellerbound} at $0$. The quadratic form $q$ is a~Dirichlet form, and
the canonically associated Markov process is a solution for the
martingale problem associated to the operator $L$ with the appropriate
killing or reflection at $0$. This means there exists a family of
measures $(\mathbb{P}_t)_{t \in(0,\infty)}$ on the space
$C([0,\infty),\mathbb{R})$ of real valued continuous functions on
$[0,\infty)$ such that for every $f \in\mL^2$ and every $x \in
(0,\infty)$ (due to the Feller property)
\[
(e^{-tL}f)(x) = \mathbb{E}_x[f(X_t),T_0 > t],
\]
where $(X_t)$ is the canonical process on $C([0,\infty),\mathbb{R})$,
and $T_0$ is a random time defined with respect to the local time at 0.
(Again, if $\alpha=\infty$, then $T_{0}$ is the time of first hitting
0; if $\alpha=0$, then $T_{0}\equiv\infty$.) In this normalization,
the scale measure has density $\rrho(x)^{-1}$ with respect to Lebesgue measure.

It is a trivial consequence of the definition of natural scale that
$\int^{X_{t}}\rrho(x)^{-1}$ is a martingale, and so that $\mathbb
{P}_x (X_{t} \mbox{ hits 0 eventually} ) = 1$ for $x>0$ if
and only if the scale function is infinite at $\infty$; that is, for
$c>0$, $\int_c^{\infty}\rrho(x)^{-1}\,dx = \infty$. When there is
killing at 0, the process is recurrent only when the scale function is
infinite at both ends. In analytic terms, recurrence means that the
associated generator is critical [see \citet{GZ91} and \citet{rP95}].
Recall that $L^{\kappa}$ is called critical iff there exists a unique
(up to constant multiples) positive solution $\psi$ of $L^{\kappa
}\psi= 0$. Otherwise $L^{\kappa}$ is called subcritical. We know from
criticality theory---for example, from Theorem~3.15 of \citet{GZ91}---that the generator must be critical if 0 is an isolated
eigenvalue. A generalization of this fact will be used in Lemma~\ref{spectrum}.

The semigroup $e^{-tL^{\kappa}}$ has a probabilistic representation:
We consider the product space
\[
C([0,\infty)) \times[0,\infty) = \lbrace(\omega,\xi) \in
C([0,\infty))\times[0,\infty) \rbrace
\]
endowed with the natural product $\sigma$-field. Let $(\tilde{\mathbb
{P}}_x)_{x \in(0,\infty)}$ denote the family of measures which is
induced by the Dirichlet form $q^0$. For $x \in(0,\infty)$ we define
the measures
\[
\tilde{\mathbb{P}}_x \otimes e^{-\xi}\,d\xi
\]
and the stopping time
\[
T_{\kappa}(\omega,\xi) = \inf \biggl\lbrace s \geq0 \Big|\int_0^s
\kappa(\omega_s)\,ds \geq\xi \biggr\rbrace.
\]
If we set
\[
\tau_{\partial} = \min ( T_0 , T_{\kappa}  )
\]
then we have the Feynman--Kac representation,
%
%e2.15 ###
%e2.15 #&#
\begin{equation}\label{FKrepresent}
\qquad (e^{-tL^{\kappa}}f)(x) = \tilde{\mathbb{E}}_x [f(X_t),\tau
_{\partial}>t ] = \mathbb{E}_x \bigl[e^{-\int_0^{t}\kappa
(X_s)\,ds}f(X_t),T_0 > t \bigr].
\end{equation}
It is easy to see that $e^{-tL^{\kappa}}$ is an integral operator. We
denoty by $p^{\kappa}(t,x,y)$ its integral kernel with respect to the
measure $\Gamma$, that is,
\[
e^{-tL^{\kappa}}f(x) = \int_0^{\infty}p^{\kappa}(t,x,y)f(y)\Gamma
(dy) \qquad \mbox{for every $f \in L^2((0,\infty),\Gamma)$}.
\]
Since we are working with the self-adjoint version of the generator
(with respect to the measure $\Rho$), the Feynman--Kac representation
holds in great generality, following the derivation in \citet{DvC00}.
We will generally omit the tilde, since it will be clear from context
which measure is meant.

Let us recall the usual Feller classification [see, e.g., Chapter 3 in
\citet{BL07}] of boundary points for diffusion generators $-\frac
{1}{2}\frac{d^2}{dx^2} - b(x)\frac{d}{dx}$ in an open interval
$(0,r)$.\vspace*{-2pt}
%
%de2.3 #&#
\begin{definition}
Let $c \in(0, r)$ be given and set $\rrho(x) = e^{\int_c^x 2b(y)\,
dy}$. The point $r$ is called \textit{accessible}, if $\int
_c^{r}\rrho(x)^{-1} \int_c^x\rrho(y)\,dy\,dx < \infty$, and
otherwise \textit{inaccessible}. If $r$ is an accessible boundary
point, then it is called \textit{regular} iff $\int_c^{r}\rrho
(x)\int_c^x\rrho(y)^{-1}\,dy\,dx < \infty$. If $r$ is accessible and
$\int_c^{r}\rrho(x)\int_c^x\rrho(y)^{-1}\,dy\,dx = \infty$, then
$r$ is called an \textit{exit boundary}. If $r$ is inaccessible, then
it is an \textit{entrance boundary}, iff $\int_c^{r}\rrho(x)\int
_c^x\rrho(y)^{-1}\,dy\,dx < \infty$.
If $r$ is inaccessible and $\int_c^{r}\rrho(x)\int_c^x\rrho
(y)^{-1}\,dy\,dx = \infty$, then $r$ is called \textit{natural}. Of
course the same classification holds for $0$.\vspace*{-2pt}
\end{definition}

Except where otherwise indicated, we will always assume that the
boundary point $\infty$ is inaccessible.\vadjust{\goodbreak}

It is easy to check that the boundary point $r$ is regular if and only
if $\int_c^{r}\rrho(x)\,dx < \infty$ and $\int_c^{r}\rrho(x)^{-1}\,
dx < \infty$. A boundary point is thus regular in the sense of Feller
if and only if it is regular in the sense of Weyl; cf. \citet{JR76}.
Let us recall the relevant definition from the Weyl theory of
self-adjoint extensions of singular Sturm--Liouville operators
$L^{\kappa} =- \frac{1}{2\rrho}\frac{d}{dx}(\rrho\frac
{d}{dx})+\kappa$ in $(0,r)$, adapted to our special
situation.\vspace*{-2pt}
%
%de2.4 #&#
\begin{definition}\label{lplc}
We say that boundary $r$ is of \textit{limit-point type}, if there
exists $c \in(0,r)$ and $z\in\mathbb{C}$, and a solution $f$ of
$(L^{\kappa}-z)f = 0$ such that $\int_c^{r}|f(y)|^2\rrho(y)\,dy =
\infty$. If there exists $c \in(0,\infty)$, such that for every
solution of the equation $(L^{\kappa}-z)f = 0$ the integral $\int
_{c}^{r}|f(y)|^2\rrho(y)\,dy$ is finite, then we say that $r$ is of
\textit{limit-circle type}. An analogous notation applies to the
boundary point~$0$.\vspace*{-2pt}
\end{definition}

A fundamental result in the theory of Sturm--Liouville operators is the
so called Weyl-alternative, which states that exactly one of the above
situations holds and that the limit-point/limit-circle classification
is independent of $z \in\mathbb{C}$ [see \citet{JR76}]. Moreover if
we are in the limit-point case at $r$, then for every $z \in\mathbb
{C}\setminus\mathbb{R}$ there exists exactly one solution of the
equation $(L^{\kappa}-z)f = 0$ which satisfies $\int_c^r|f(s)|^2
\rrho(y)\,dy< \infty$. Roughly limit-circle case at a boundary point
$r$ means that we have to specify boundary conditions at $r$ in order
to get a self-adjoint realization, whereas in the limit-point case at
$r$ no boundary conditions at $r$ are necessary.\vspace*{-2pt}

%s2.3 ###
%s2.3 #&#
\subsection{Quasi-limiting and quasi-stationary behavior} \label{sec:quasi}
We say that $X_t$ \textit{converges from the initial distribution $\nu
$ to the quasistationary distribution $\varphi$ on compacta} if for
any positive $z$, and any Borel $A \subset[0,z]$
\[
\lim_{t\rightarrow\infty} \mathbb{P}_{\nu} ( X_t \in A |
X_t \leq z) =\frac{\int_A\varphi(y) \rrho(y)\,dy}{\int_0^{z}\varphi
(y) \rrho(y)\,dy};
\]
$X_t$ \textit{converges from the initial distribution $\nu$ to the
quasistationary distribution $\varphi$} if $\int_0^{\infty}\varphi
(y) \rrho(y)\,dy<\infty$, and for any Borel subset $A \subset
[0,\infty)$
\[
\lim_{t\rightarrow\infty} \mathbb{P}_{\nu} ( X_t \in A |
\tau_{\partial} > t) =\frac{\int_A\varphi(y) \rrho(y)\,dy}{\int
_0^{\infty}\varphi(y) \rrho(y)\,dy}.
\]
Finally we say that $X_t$ \textit{escapes from the initial
distribution $\nu$ to infinity} if
\[
\lim_{t \rightarrow\infty} \mathbb{P}_{\nu} ( X_t \leq z |
\tau_{\partial}> t  ) = 0.\vspace*{-2pt}
\]

%re2.5 #&#
\begin{remark}\label{qsl}
In the literature there is no completely standard terminology for
quasistationary distributions. The probability measure $\frac{\varphi
(y) \Rho(dy)}{\int_0^{\infty}\varphi(y) \rrho(y)\,dy}$ described
here is sometimes also called a quasi-limiting distribution.
A~quasistationary distribution $\tilde{\nu}$ is often defined as a
probability measure $\tilde{\nu}$ supported in $(0,\infty)$ satisfying
\[
\mathbb{P}_{\tilde{\nu}} (X_t \in A |\tau_{\partial} >
t ) = \tilde{\nu}(A)  \qquad \forall\mbox{ Borel sets } A\subset
(0,\infty), t > 0.\vadjust{\goodbreak}
\]
Quasilimiting distributions are also called Yaglom limits. It is not
difficult to see that quasilimiting distributions are also
quasistationary distributions.\vspace*{-2pt}
\end{remark}

%s2.4 ###
%s2.4 #&#
\subsection{Previous results} \label{sec:previous}
Observe that we have in equation \eqref{possolutions} that $\varphi
(\lambda_0^{\kappa},\cdot)$ is positive. \citet{quasistat} showed
a slightly weaker version of the following result. Their additional
assumptions concerning the $b$ and $\kappa$ are easily seen to be unnecessary.\vspace*{-2pt}
%
%th2.6 #&#
\begin{theorem}[{[Theorem 3.3 in \citet{quasistat}]}]\label{Thm33}
Assume that $\infty$ is a natural boundary point and that we are in
the limit-point case at $\infty$. Suppose that either
\[
\liminf_{x \rightarrow\infty}\kappa(x) > \lambda_0^{\kappa}
\qquad \mbox{or} \qquad  \limsup_{x\rightarrow\infty}\kappa(x) < \lambda
_0^{\kappa}.
\]
Then either $X_t$ converges to the quasistationary distribution
$\varphi(\lambda_0^{\kappa},y)\,d\gamma(y)/\break\int_0^{\infty}\varphi
(\lambda_0^{\kappa},y)\,d\gamma(y)$, or $X_t$ escapes to infinity. In
the case $\liminf\kappa(x) > \lambda_0^{\kappa}$, $X_t$~converges
to the quasistationary distribution $\varphi(\lambda_0^{\kappa
},\cdot)$ if and only if\break $\int_0^{\infty}\varphi(\lambda_0^{\kappa
},y) \times\rrho(y)\,dy$ is finite.\vspace*{-2pt}
\end{theorem}

A priori it would not have been clear that the conditional distribution
converges, and that the mass cannot split, with part of the mass
remaining on a compact interval and the remainder escaping to infinity.
Having recognized that there is a a dichotomy, it is natural to then
seek a simple criterion for discriminating between the cases: escape or
convergence. One such is given in \citet{quasistat}, under which
$X_t$ converges to quasistationarity, namely when $\lambda_0^{\kappa}
< K =: \lim_{t \rightarrow\infty}\kappa(t)$ together with the
growth bound
{\renewcommand{\theequation}{$\mathit{GB}'$}
%e2.16 #&#
\begin{equation}\label{eqGB1}
\exists\tilde{b}, \tilde{\kappa}\geq0\
\forall y \mbox{ large enough: } |b(y)| \leq\tilde{b}y   \mbox{
and }  \kappa(y)\leq\tilde{\kappa}y
\end{equation}}

\vspace*{-\baselineskip}

\noindent or the related bound
{\renewcommand{\theequation}{$\mathit{GB}''$}
%e2.17 #&#
\begin{eqnarray}
 &&\exists\bar{b}_1, \bar{b}_2, \bar{\kappa},
\beta\geq0  \ \forall y \mbox{ large enough: }\bar{b}_1y^{\beta
}\geq b(y)\geq-\bar{b}_1y, b'(y)\geq-\bar{b}_2y^2\nonumber\hspace*{-35pt}\\[-8pt]\\[-8pt]
 && \hphantom{\exists\bar{b}_1, \bar{b}_2, \bar{\kappa},
\beta\geq0  \ \forall y \mbox{ large enough: }}\mbox{and } \kappa(y)\leq
\tilde{\kappa}y.\nonumber\hspace*{-35pt}
\end{eqnarray}}

\vspace*{-\baselineskip}

\noindent While these conditions are satisfied in many applications they are,
from a~theoretical point of view, unsatisfactory. In particular, it
seems peculiar that an upper bound on the killing rate as in (\ref{eqGB1})
should be necessary. On the contrary increasing the killing rate
$\kappa$ should, from a heuristic point of view, only strengthen the
convergence to quasistationarity.\vspace*{-2pt}
%
%re2.7 #&#
\begin{remark}
We make use of Theorem \ref{Thm33} only in the case
\[
\lambda
_0^{\kappa} > \lim_{x \rightarrow\infty}\kappa(x)
\]
and $\Rho ((0,\infty)) < \infty$. In the other cases we use different
techniques. In the next chapter we will show that $\infty$ is always
in the limit-point case. As emphasized and explained in \citet{quasistat} in this case the heuristic behind Theorem~\ref{Thm33} is
quite clear, but the translation of this idea into formal mathematics
is not trivial.\vadjust{\goodbreak}
\end{remark}

%s3 ###
%s3 #&#
\section{Analytic results}
In this chapter we derive several key analytic facts about the spectra
of generators and resolvents. While some of these are standard in the
theory of Sturm--Liouville operators, and well known to specialists in
that field, they are less familiar to probabilists, and we explain them
in some detail here. In Section \ref{sec:FW} we show that the
technical conditions for a limit-point boundary at~$\infty$ may be
weakened. Section \ref{sec:SLspec} derives basic results linking the
spectrum and speed measure. Section \ref{sec:LPHI} presents the
standard parabolic Harnack inequality in the form that we will be
using. Section \ref{sec:SRLT} applies the analytic results to
convergence on compacta. Section~\ref{sec:ID} explains why strong
conditions on the initial conditions are unnecessary. Finally, Section
\ref{sec:entrance} generalizes the results to the case of an entrance
boundary at $\infty$.

%s3.1 ###
%s3.1 #&#
\subsection{Classification of boundary points} \label{sec:FW}
We start by establishing a connection between the Feller classification
and the Weyl classification of boundary points. This has already been
investigated in \citet{nW85} for the case $\kappa= 0$, but in this
work the author introduces the notion of weak entrance boundary and
shows that one is in the limit-circle case if the boundary point is of
weak entrance type. We show that there are no weak entrance boundaries
at $\infty$ by proving that $\infty$ is in the limit-point case. The
proof we give is well known in the Schr\"{o}dinger case [see \citet{BMS02} for similar ideas in a much more general context]. We assume
regularity of the coefficients of the Sturm--Liouville expression,
although weaker assumptions would also suffice.
%
%le3.1 #&#
\begin{lemma} \label{L:SL}
Let the Sturm--Liouville expression $\tau f(x) = -\frac{1}{\rrho
(x)} (\rrho(x)\times f'(x) )' + \kappa(x)f(x)$ be given. Assume
that $\rrho$ is strictly positive and locally Lipschitz in $(0,\infty
)$ and $\kappa\in\mL^2_{loc}([0,\infty))$ such that $\kappa(x)
\geq-C|x|^2+D$ for some constants $C,D\geq0$. Then we are in the
limit-point case at $\infty$.
\end{lemma}

\begin{pf}
We can assume, without loss of generality, that $D=0$ and that~$\rrho$
is continuous up to the boundary. The first assumption is obviously
harmless. If $\rrho$ is not continuous up to zero we can consider the
differential expression in $(1,\infty)$ instead of $(0,\infty)$. This
shift does not change the Weyl-classification of $\tau$ at infinity.
Similarly, we may assume that the boundary condition at 0 is Dirichlet
($\alpha=\infty$), since the classification at infinity is unaffected
by the boundary condition at 0.

As usual in the theory of Sturm--Liouville operators we define the
maximal operator $T$ and the minimal operator $\widetilde{T}$
associated to the differential expression $\tau$ as
\begin{eqnarray*}
\mathcal{D}(T)&:=& \lbrace f \in\mL^2 | f, \rrho f' \mbox{
absolutely continuous in $(0,\infty)$, } \tau f \in\mL^2
\rbrace,
\\
Tf&:=&\tau f \quad \mbox{for $f \in\mathcal{D}(T)$}
\end{eqnarray*}
and
\begin{eqnarray*}
\mathcal{D}(\widetilde{T})&:=& \lbrace f \in\mathcal{D}(T)|
  f\mbox{ has compact support in $(0, \infty)$} \rbrace,\\
\widetilde{T}f&:=& \tau f \quad \mbox{for $f \in\mathcal{D}(\widetilde{T})$},
\end{eqnarray*}
respectively. Let $T_{D}$ be the restriction of the maximal operator
$T$ to the domain
\[
\mathcal{D} =  \lbrace f \in\mathcal{D}(T) | f(0) = 0
\rbrace,
\]
that is, we put Dirichlet boundary conditions at the boundary point $0$.

The deficiency indices (we refer to the short summary in the \hyperref[app]{Appendix})
of~$\widetilde{T}$ are $(1,1)$ if the limit-point case holds at
$\infty$, and $(2,2)$ if limit-circle holds at $\infty$. In the
former case, the maximal symmetric (self-adjoint) extensions of
$\widetilde{T}$ are one-dimensional; in the latter case, they are
two-dimensional. If~$T_{D}$ defines a symmetric operator---$\langle
f,T_{D}f\rangle\in\mathbb{R}$ for every $f \in\mathcal{D}$---then
it cannot have dimension higher than 2. But there is one free parameter
at 0; in the limit-circle case there would be two free parameters at
$\infty$. Thus, in the limit-circle case $T_{D}$ would be a
three-dimensional extension of $\widetilde{T}$, so it could not be
symmetric. If we show that $T_{D}$ is symmetric, it will follow that
the Sturm--Liouville problem is in the limit-point case at $\infty$.

Let $\varphi\in C^{\infty}_c(\mathbb{R})$ such that $0\leq\varphi
\leq1$ and
\[
\varphi(x) =
\cases{
1, & \quad if $|x|\leq1$,\cr
0, & \quad if $|x| \geq2$.
}
\]
Further we set $\varphi_k(x) = \varphi(\frac{x}{k})$ ($k \in\mathbb
{N}$). This gives, for $f \in\mathcal{D}(T_{D})$ and $k \in\mathbb{N}$
%
%e3.6 ###
%e3.5 ###
%e3.4 ###
%e3.3 ###
%e3.2 ###
%e3.1 ###
%e3.1 #&#
\begin{eqnarray}\label{intparts1}
\langle f,T_{D}f\rangle&=& \lim_{k\rightarrow\infty}\int_0^{\infty
}\varphi_k(x)^2\overline{f(x)}T_{D}f(x) \rrho(x)\,dx \nonumber \\
&=& \lim_{k\rightarrow\infty} \int_0^{\infty}\varphi
_k(x)^2\overline{f(x)} \biggl[-\frac{1}{2\gamma}(\gamma f')'(x)+
\kappa(x)f(x) \biggr]\gamma(x)\,dx \\
&=& \lim_{k\rightarrow\infty} \biggl[\frac{1}{2}\int_0^ {\infty
} (\varphi_k(x)^{2}\overline{f(x)} )'f'(x)\rrho(x)\,dx \nonumber\\
&&\hphantom{\lim_{k\rightarrow\infty} \biggl[}
{} + \int_0^{\infty}\varphi_k(x)^{2}\kappa
(x)|f(x)|^2\rrho(x)\,dx \biggr]\nonumber\\
&=& \lim_{k\rightarrow\infty}  \biggl\lbrace\int_0^{\infty}\varphi
_k^2(x) \biggl(\frac{1}{2}|f'(x)|^2 + \kappa(x)|f(x)|^2 \biggr)\,\rrho
(x)\,dx \nonumber \\
&&\hspace*{69pt}{} + \int_0^{\infty}\varphi_k(x)\varphi_k'(x)\overline
{f(x)}f'(x)\rrho(x)\,dx \biggr\rbrace.\nonumber
\end{eqnarray}
Observe that in the third line the boundary term $\varphi_k^{2}\bar
{f}\rrho f'|_0^{\infty}$ coming from the integration by parts
vanishes, since $\rrho f$ is continuous up to $0$ (since it satisfies
an ODE), $f(0)=0$ and $\varphi_k(x)$ is identically zero for $x$ large enough.
%Note: More general boundary conditions will still make this term real.

The first term on the right-hand side is real, and we have to prove
that the second term converges to $0$ as $k \rightarrow\infty$. We
have, by the Cauchy--Schwarz inequality and the properties of the
cut-off sequence $(\varphi_k)$,
%
%e3.7 ###
%e3.2 #&#
\begin{eqnarray}\label{CauchySchwarzself}\qquad
&& \biggl|\int_0^{\infty}\varphi_k(x)\varphi_k'(x)\overline
{f(x)}f'(x)\rrho(x)\,dx  \biggr| \nonumber \\
&&\qquad \leq \biggl(\int_0^{\infty}\varphi_k(x)^2|f'(x)|^2\rrho(x)\,dx\int
_0^{\infty}|\varphi_k'(x)|^2|f(x)|^2 \rrho(x)\,dx \biggr)^{
{1}/{2}}\\
&&\qquad \leq Ck^{-1} \biggl(\int_0^{\infty}\varphi_k(x)^2|f'(x)|^2 \rrho
(x)\,dx\int_k^{2k}|f(x)|^2\rrho(x)\,dx \biggr)^{{1}/{2}}.\nonumber
\end{eqnarray}
For the first integral on the right-hand side we integrate by parts in
a similar vein to \eqref{intparts1}. The assumptions on $\kappa$ as
well as the elementary inequality $|ab| \leq a^{2}/4 + b^2$ imply
\begin{eqnarray*}
&&\frac{1}{2}\int_0^{\infty}\varphi_k(x)^2|f'(x)|^2\rrho(x)\,dx \\
&&\qquad =
\int_0^{\infty}\varphi_k(x)^{2}\overline{f(x)}(T_{D}f(x)) \rrho
(x)\,dx \\
&&\qquad \quad {} - \int_0^{\infty}\varphi_k(x)^2\kappa(x)|f(x)|^2 \rrho
(x)\,dx - \int_0^{\infty}\varphi_k\varphi_k'(x)\overline{f(x)}f'(x)
\rrho(x)\,dx\\
&&\qquad  \leq\int_0^{\infty}\varphi_k(x)|f(x)||T_{D}f(x)| \rrho
(x)\,dx + C \int_0^{\infty}\varphi_k^2x^2|f(x)|^2 \rrho(x)\,dx\\
&&\qquad \quad {}  +\int_0^{\infty}|\varphi_k(x)\varphi
_k'(x)\overline{f(x)}f'(x)| \rrho(x)\,dx\\
&&\qquad  \leq \int_0^{\infty}\varphi_k(x)|f(x)||T_{D}f(x)| \rrho
(x)\,dx + C \int_0^{\infty}\varphi_k^2x^2|f(x)|^2 \rrho(x)\,dx\\
&&\qquad \quad {}  +\int_0^{\infty} \biggl[\frac{1}{4}|\varphi
_k(x)f'(x)|^2+ |\varphi_k'(x)\overline{f(x)}|^{2} \biggr] \rrho
(x)\,dx\\
&&\qquad  \leq \|f\|\|T_{D}f\| + C(2k)^2\|f\|^2\\
& &\qquad \quad {} +\frac{1}{4}\int_0^{\infty}\varphi
_k(x)^2|f'(x)|^2 \rrho(x)\,dx + M^2 k^{-2}\|f\|^2.
\end{eqnarray*}
This yields
\[
\int_0^{\infty}\varphi_k(x)^2|f'(x)|^2 \rrho(x)\,dx \leq4\|f\|\|
T_{D}f\| + 16Ck^2\|f\|^2 + 4 M^2k^{-2}\|f\|^2
\]
and therefore for large $k$
%
%e3.8 ###
%e3.3 #&#
\begin{equation}\label{growthink}
\int_0^{\infty}\varphi_k(x)^2|f'(x)|^2\,dx \leq C_1 + C_2k^2 \leq C_3k^{2}.
\end{equation}
Thus inequalities \eqref{CauchySchwarzself} and \eqref{growthink}
imply that [observe that $f \in L^2((0,\infty),\rrho)$]
\begin{eqnarray*}
 \biggl|\int_0^{\infty}\varphi_k(x)\varphi_k'(x)\overline
{f(x)}f'(x) \rrho(x)\,dx \biggr| &\leq& Ck^{-1} \biggl(C_3k^2\int
_k^{2k}|f(x)|^2 \rrho(x)\,dx \biggr)^{{1}/{2}}\\
&\rightarrow&0,
\end{eqnarray*}
as $k\rightarrow\infty$. This proves the assertion, and so completes
the proof.
\end{pf}

%s3.2 ###
%s3.2 #&#
\subsection{The spectrum of Sturm--Liouville operators} \label{sec:SLspec}
We begin with a version of the spectral theorem for self-adjoint
operators on a Hilbert space, specifically adapted to Sturm--Liouville
operators. A proof of it can be found in general references on the
theory of Sturm--Liouville or Schr\"{o}dinger operators, such as \citet{GZ06}, \citet{CL90}, \citet{aZ05}.

Let $\tau=-\frac{1}{2\rrho}\frac{d}{dx} (\rrho\frac{d}{dx})
+ \kappa$ be a Sturm--Liouville expression which is regular at $0$ and
in the limit-point case at infinity, and let $H$ be the self-adjoint
realization of $\tau$ in $\mL^2$ with boundary conditions \ref
{E:fellerbound} at $0$. Let $\varphi(z,\cdot)$ be the unique solution
of the ordinary differential equation $\tau\varphi(z,\cdot)=z\varphi
(z,\cdot)$ satisfying $\varphi(z,0)=1/(1+\alpha)$ and $\frac
{1}{2}\varphi'(z,0)=\alpha/(1+\alpha)$.

Given a continuous function $F\in C(\mathbb{R})$ and a $\sigma
$-finite measure $\mu$ on $\mathbb{R}$, we have a corresponding
maximal multiplication operator $M_{F}$ on $\mL^{2}(\mathbb{R}, \mu
)$ defined by
\begin{eqnarray*}
\mathcal{D}(M_{F})&=& \{g\in\mL^{2}(\mathbb{R}, \mu) \mbox{
s.t. }gF\in\mL^{2}(\mathbb{R}, \mu) \},\\
M_{F}(g)&=&Fg.
\end{eqnarray*}

%th3.2 #&#
\begin{theorem}[(Weyl's spectral theorem)] \label{weylspectraltheorem}
There exists a measure $\sigma$ whose support is $\Sigma(H)$, such
that the map taking a compactly supported function $h\in\mL
^2((0,\infty),\Gamma) $ to the function $\hat{h}\in\mL^2(\Sigma
(H), \sigma)$, defined by
\[
\hat{h}(\cdot)= \int_0^{\infty}h(x)\varphi(\cdot,x) \rrho(x)\,dx
\]
may be uniquely extended to a unitary mapping $U\dvtx\mL^{2}((0,\infty
),\Rho)\to\mL^{2}(\Sigma(H),\sigma)$ with the property
\[
U F(H)U^{-1} = M_F.
\]
The spectrum of $H$ is simple, and $\Sigma(F(H)) = \operatorname{ess\,
ran}_{\sigma}(F)$.
\end{theorem}

The spectrum $\Sigma(A)$ of a self-adjoint operator $A$ may be divided
into two components: the essential spectrum $\Sigma_{\mathrm{ess}}(A)$,
comprising the limit points and eigenvalues of infinite multiplicity,
and the discrete part $\Sigma_d(A)$, comprising the isolated
eigenvalues of finite multiplicity. In the Sturm--Liouville case every
eigenvalue has finite multiplicity (no more than 2), so the essential
spectrum consists only of limit points of the spectrum. It is well
known that the essential part of the spectrum of self-adjoint operators
is invariant with respect to relatively compact perturbations [see
Theorem 9.15 of \citet{jW00}]. [We recall that an operator
$V\dvtx X\rightarrow X$ on the Banach space $X$ is called relatively compact
with respect to $T\dvtx X\to X$ if $\mathcal{D}(T) \subset\mathcal
{D}(V)$, and if for some $z \in\mathbb{C}\setminus\Sigma(T)$ the
operator $V(T-z)^{-1}$ is compact. We refer to Section 9.2 of
\citet{jW00} for further details.]

The core of our results is contained in the following analytic lemma,
which catalogs some of the key linkages among the base of the spectrum,
the scale measure and the speed measure. These take us beyond the
results of Theorem \ref{Thm33}, by separating the influence of the
drift from the effect of the killing term. Moreover, they show clearly
why the case $\lambda_0^{\kappa} < K$ will turn out to be easier than
the case $\lambda_0^{\kappa}>K$. The major results---particularly
Theorems \ref{generallocalMandl}, \ref{limkbiggerev}, \ref{qsddrift}
and \ref{Escape}---will in essence be just unpacking these analytic
results in probabilistic terminology.

%le3.3 #&#
\begin{lemma}\label{spectrum}
With the above definitions:
\begin{enumerate}[(vii)]
\item[(i)] if $\lim_{x \rightarrow\infty}\kappa(x) = K$, then
$\Sigma_{\mathrm{ess}}(L^{0,\alpha})+K = \Sigma_{\mathrm{ess}}(L^{\kappa,\alpha
})$;\hypertarget{it:specess1}{}
\item[(ii)] $\lambda_0^{0,\infty} > 0$ and $\int_0^{\infty}\rrho
(x)^{-1}\,dx=\infty$ imply $\Rho(\mathbb{R}_+) =\izf\rrho(x)\,dx <
\infty$; \hypertarget{it:speedfinite}{}
\item[(iii)] $\lambda_0^{0,\alpha} > 0$ and $\Rho ([0, \infty
) ) = \infty$ imply $\lambda_{0}^{0,0}>0$;\hypertarget{it:neumannbottom}{}
\item[(iv)] $\lambda_0^{0,\alpha} > 0$ and $\Rho ([0, \infty
) ) = \infty$ imply $\lim_{r \rightarrow\infty}\frac
{1}{r}\log\Rho([0,r)) > 0$;\hypertarget{it:speedexpon}{}
\item[(v)] if $\lambda_0^{\kappa,\alpha} < \liminf_{x\rightarrow
\infty}\kappa(x)$, then $\lambda_0^{\kappa}$ is a simple isolated
eigenvalue with a unique positive eigenfunction;\hypertarget{it:isolated}{}
\item[(vi)] if $\alpha>0$ (not pure reflection at 0) or if $ \Gamma
([0,\infty)) = \infty$, then 0 is not an isolated eigenvalue of
$L^{0,\alpha}$;\hypertarget{it:noisolated}{}
\item[(vii)] if $\alpha>0$ (not pure reflection at 0) or if $ \Gamma
([0,\infty)) = \infty$, then $\lambda_0^{\kappa} > \limsup_{x
\rightarrow\infty}\kappa(x)$ implies $\lambda_0^{0,\alpha} > 0$.\hypertarget{it:lambda0}{}
\end{enumerate}
\end{lemma}

\begin{pf}
 {Assertion} \hyperlink{it:specess1}{(i)} can be derived from the fact that
the essential spectra of two self-adjoint operators $L_1$ and $L_2$
coincide if for some $z \in\mathbb{C}\setminus(\Sigma(L_1)\cup
\Sigma(L_2))$ the difference
\[
(L_1-z)^{-1}-(L_2-z)^{-1}
\]
is a compact operator; cf. Theorem 9.15 of \citet{jW00}.
Set $\kappa_n(t) =\mathbf{1}_{[0,n]}(t)(\kappa(t)-K)$. The resolvent
equation gives for $z \in\mathbb{C}\setminus\mathbb{R}$
\begin{eqnarray*}
(L^{\kappa_n}-z)^{-1} - (L-z)^{-1} &=& (L^{\kappa
_n}-z)^{-1}(L-L^{\kappa_n})(L-z)^{-1} \\
&=& -(L^{\kappa_n}-z)^{-1}\kappa_n(L-z)^{-1}.
\end{eqnarray*}
Observe now that the operator $\kappa_n(L-z)^{-1}$ is compact; that
is, the operator acting by multiplication with $\kappa_n$ is
relatively compact with respect to the operator $L$. This can be seen
by considering the explicit form of the resolvent [see Chapter~3.3 in
\citet{JR76}; similar results can be found in \citet{CL55}]. We have
\begin{eqnarray*}
&&[\kappa_n(L-z)^{-1}]g(x) \\
&&\qquad = \kappa_n(x) \frac{1}{W(v,u)} \biggl(
v(x)\int_0^x u(y)g(y)\rrho(y)\,dy + u(x)\int_x^{\infty}v(y)g(y)\rrho
(y)\,dy \biggr)
\end{eqnarray*}
where $u$ and $v$ are linearly independent solutions of
%
%e3.4 #&#
\begin{eqnarray}
(\tau-z)w = 0 \mbox{ satisfying }\nonumber \\
\eqntext{\displaystyle u(0)=\frac{1}{1+\alpha},   \frac{1}{2}u'(0)=\frac{\alpha
}{1+\alpha} \mbox{ and }
\int_1^{\infty} |v(y)|^2 \rrho(y)\,dy<\infty.}
\end{eqnarray}
Observe that here we use the fact that we are in the limit-point case
at infinity.
The Wronskian $W(f,g)$ of two locally absolutely continuous functions~$f$
and $g$ is defined by
\[
W(f,g)(x) =  [f(x) g'(x)- f'(x)g(x) ]\rrho(x).
\]
Thus $\kappa_n(L-z)^{-1}$ is an integral operator in $\mL^2$ with
kernel $k(\cdot,\cdot)$ given by
\[
k(x,y) =
\cases{
W(v,u)^{-1}\kappa_n(x)v(x)u(y), &\quad  if $ y \leq x$, and \cr
W(v,u)^{-1}\kappa_n(x)v(y)u(x), &\quad  if $ y \geq x$.
}
\]
The known properties of $u$ and $v$ imply
\begin{eqnarray*}
&&\int_0^{\infty}\int_0^{\infty}|k(x,y)|^2 \rrho(y) \rrho(x)\,dy
\,dx \\
&&\qquad = \frac{1}{W(v,u)^2}\int_0^{n} \biggl(|v(x)|^2\int_0^x|u(y)|^2
\rrho(y) \,dy + |u(x)|^2\int_x^{\infty}|v(y)|^2 \rrho(y)\,dy \biggr)\\
&&\qquad \quad \hphantom{\frac{1}{W(v,u)^2}\int_0^{n}}
{}\times|\kappa_n(x)|^{2}\rrho(x)\,dx \\
&&\qquad < \infty.
\end{eqnarray*}
Thus $\kappa_n(L-z)^{-1}$ is Hilbert--Schmidt, hence also compact.

We complete the proof by observing that the resolvent equation
\[
(L^{\kappa_n+K}-z)^{-1} - (L^{\kappa}-z)^{-1} = (L^{\kappa
_n}-z)^{-1}(\kappa- \kappa_n-K)(L^{\kappa}-z)^{-1},
\]
implies
\begin{eqnarray*}
 \|(L^{\kappa_n+K}-z)^{-1} - (L^{\kappa}-z)^{-1} \| &\leq&\|
(L^{\kappa_n}-z)^{-1}\| \|\kappa-\kappa_n-K\|_{\infty} \|(L^{\kappa
}-z)^{-1}\| \\
&\leq&\frac{1}{(\Im z)^2} \|\kappa-\kappa_n-K\|_{\infty}
\rightarrow0
\end{eqnarray*}
as $n\rightarrow\infty$; that is, $L^{\kappa_{n}+K}$ converges in
the norm-resolvent sense to $L^{\kappa}$. In the second inequality we
used the fact that the operator norm of the operator that acts as
multiplication by a function $f$ is just the supremum norm of
$f$.

 {Assertion} \hyperlink{it:speedfinite}{(ii)} is contained in \citet{MSM01}
and also follows from Theorem 1 of the recent work \citet{rP09}.
[In \citet{rP09} somewhat stronger conditions on the drift are imposed,
but these are not actually necessary for the proof.]

 {Assertion} \hyperlink{it:neumannbottom}{(iii)}: Because $L^{0,\alpha}$ and
$L^{0,0}$ differ only in their (one-dimensional) boundary conditions,
the difference $(L^{0,0}+1)^{-1}-(L^{0,\alpha}+1)^{-1}$ has
one-dimensional range, so it is compact. For more details, see
\citeauthor{jW00} (\citeyear{jW00}), Satz~10.17. The bottom of the essential spectrum of $L^{0,0}$
is then strictly positive, since it coincides with the bottom of the
essential spectrum\break of~$L^{0,\alpha}$, hence is above the bottom of the
full spectrum of $L^{0,\alpha}$. If $\lambda_0^{0,0}:=\inf\on
{spec}(L^{0,0})=0$, then $\lambda_0^{0,0}=0$ is necessarily an
isolated eigenvalue of the operator $L^{0,0}$. Let us assume that
$\lambda_0^{0,0} = 0$. The unique (up to positive multiples)
nontrivial and nonnegative eigenfunction $v_{N} \in\mL^2$
associated to $\lambda_0^{0,0}=0$ therefore solves the boundary value problem
\[
L^{0,0} v_{N} = \lambda_0^{0,0} v_{N} = 0,  \qquad   v_{N}(0) =\frac
{1}{1+\alpha} \quad \mbox{and}\quad   \frac{1}{2}\frac{dv_{N}}{dx}(0) =
\frac{\alpha}{1+\alpha}.
\]
Since this ordinary differential equation has a unique solution, and
since the constant function $\mathbf{1}$ is also a solution of this
equation, we conclude that $v_N = \mathbf{1}$. Thus $\mathbf{1}\in
\mL^2$, which means that $\Rho ((0,\infty) )<\infty$,
contradicting our assumption that $\Rho$ is infinite. It follows that
$\lambda_0^{0,0}>0$.

 {Assertion} \hyperlink{it:speedexpon}{(iv)} follows from the above and the
work of \citet{lN98}. His result implies that the bottom
of the essential spectrum of the operator~$L^{0,0}$ is bounded above by
$\limsup_{r \rightarrow\infty}\frac{1}{r}\log\Rho((0,r))$. This
is $0$ if the volume growth is subexponential. Since we have already
showed that $\lambda_{0}^{0,0}>0$, the result follows.

 {Assertion} \hyperlink{it:isolated}{(v)}: Assume first that $\lim_{x
\rightarrow\infty}\kappa(x)$ exists. If $\lambda^{\kappa}_0 < \lim
_{x\to\infty}\kappa(x) = K$, then an application of the result \hyperlink{it:specess1}{(i)} shows that $L^{\kappa} = L + K + (\kappa-K)$ has the
same essential spectrum as $L + K$. Since $L$ is a positive operator,
the bottom of the spectrum of $L + K$, hence a fortiori of the
essential spectrum, has to be at least $K$, hence bigger than $\lkap$,
which implies \hyperlink{it:isolated}{(v)}.

Let us now assume only that $\liminf_{x \rightarrow\infty}\kappa
(x)>\lambda_0^{\kappa}$. By the decomposition principle [see Section
131 in \citet{AG81}]
it is not difficult to see that $L^{\kappa}$ has the same essential
spectrum as the operator $L^{\kappa}_a$ ($a>0$), defined as the
self-adjoint extension of $\tau^{\kappa}$ in $\mL^2((a,\infty),\Rho
)$ satisfying Dirichlet boundary conditions at $a$.
If $a_0>0$ and $\varepsilon>0$ are such that $\inf_{x\geq a_0}\kappa
(x)>\lambda_0^{\kappa}+\varepsilon$ we conclude that
\begin{eqnarray*}
\inf\Sigma_{\mathrm{ess}}(L^{\kappa}) &\geq&\inf\Sigma(L_{a_{0}}^{\kappa})
\\
&\geq&\mathop{\inf_{\varphi\in C^{\infty}_c(a_0,\infty)}}_{\|
\varphi\|_{\mL^2((a_0,\infty),\Rho)}=1} \int_{a_0}^{\infty
}|\varphi'(x)|^2 \rrho(x)\,dx + \int_{a_0}^{\infty}\kappa
(x)|\varphi(x)|^2\rrho(x)\,dx \\
&\geq&\lambda_0^{\kappa}+\varepsilon.
\end{eqnarray*}

 {Assertion} \hyperlink{it:noisolated}{(vi)}: Let $f,g$ be continuous
functions, nonnegative and not identically zero, with compact support
on $(0,\infty)$, and $\lambda_{1}:= \inf\Sigma(L^{0,\alpha
})\setminus\{0\}$, where $\alpha>0$. Suppose 0 is an isolated
eigenvalue, so that $\lambda_{1}>0$. Then
\begin{eqnarray*}
 \langle f, e^{-tL^{0,\alpha}}g \rangle&=& \int_{\Sigma
(L^{0,\alpha})} e^{-t\lambda}\,d \langle f, E g \rangle
(\lambda)\\
&=&  \langle f, E^{0,\alpha}(\{0\}) g \rangle+ \int
_{\lambda_{1}}^{\infty} e^{-t\lambda} \,d \langle f, E^{0,\alpha
} g \rangle(\lambda)\\
&\to&\langle f,u_{0}\rangle\langle g,u_{0}\rangle\qquad \mbox{as } t\to
\infty.
\end{eqnarray*}
This is positive, since $u_{0}$ may be chosen to be strictly positive.

By Lemma \ref{L:sqrtcompare} [observe that $e^{-tL^{0,\alpha}}g \in
\mathcal{D}(L^{0,\alpha})$ according to equation (3)] there is a
constant $C$ such that for all $x\in\on{supp}(g)$,
\begin{eqnarray*}
 |e^{-tL^{0,\alpha}}g(x) |&\le& C \bigl\| \sqrt{L^{0}}
e^{-tL^{0}} g \bigr\|\\
&\le& C \biggl(\int_{0}^{\infty} \lambda e^{-2\lambda t} d\|E^{0,\alpha
}_{\lambda} g\|^{2} \biggr)^{1/2}\\
&\le& C\|g\| e^{-\lambda_{1}t} \qquad \mbox{for $t$ sufficiently large},
\end{eqnarray*}
so that
\[
 \langle f, e^{-tL^{0}}g \rangle\le C\|g\| e^{-\lambda
_{1}t} \int_{0}^{\infty} f(x) \rrho(x)\,dx \to0
\]
as $t\to\infty$, which is a contradiction.

 {Assertion} \hyperlink{it:lambda0}{(vii)}: Suppose first that the limit $K$
of $\kappa(x)$ exists. Intuitively, what we are saying is that when
the mass in a neighborhood of 0 shrinks at a rate faster than~$K$ (what
$\lambda_{0}^{\kappa}$ measures), it is being driven by drift: Either
the mass is being swept down into a region of high killing near 0, or
it is being swept up away from~0. In the latter case, the drift will
still cause the mass near 0 to shrink exponentially in the absence of
killing; in the former case, the killing at 0 will do the job, except
in the case of pure reflection at~0.\looseness=-1

By part \hyperlink{it:specess1}{(i)}, we see that $L^{\kappa} = L + K +
(\kappa-K)$ and $L + K$ have the same essential spectrum. In
particular we conclude that $\inf\Sigma_{\mathrm{ess}}(L) + K = \inf\Sigma
_{\mathrm{ess}}(L+K) \geq\lambda_0^{\kappa}$ and therefore $\inf\Sigma
_{\mathrm{ess}}(L) \geq\lambda_0^{\kappa} - K > 0$, so $\lambda_{0}=0$ would
imply that $0$ is an isolated eigenvalue. Since this is impossible, by
assertion~\hyperlink{it:noisolated}{(vi)}, it follows that $\lambda_0 > 0$. The
extension to the case when the limit does not exist goes exactly the
same way as in the proof of assertion \hyperlink{it:isolated}{(v)} above.
\end{pf}

%re3.4 #&#
\begin{remark}
$\!\!$It was shown in \citet{rP09} that conclusion \hyperlink{it:speedfinite}{(ii)}
of Lemma~\ref{spectrum} can be sharpened. Assuming that absorption is
certain, it was shown that
%
%e3.9 ###
%e3.5 #&#
\begin{equation}\label{Mu}
\frac{1}{8A(b)} \leq\lambda_0 \leq\frac{1}{2A(b)},
\end{equation}
where
\[
A(b) = \sup_{x>0} \biggl(\int_{x}^{\infty} \rrho(y)\,dy \biggr)
\biggl(\int_0^x\rrho(y)^{-1}\,dy  \biggr).
\]
Related analytic inequalities, which are usually referred to as
weighted Hardy inequalities, can be found in \citet{bM72}. Indeed the
results of Muckenhoupt (\citeyear{bM72}) imply Pinsky's bounds.
\end{remark}

%re3.5 #&#
\begin{remark}
The fact that the bottom of the spectrum is an isolated eigenvalue is
also of practical interest, because in this case the associated
eigenfunction can be approximated accurately by the ground states of
regular Sturm--Liouville operators on bounded intervals [see the recent
survey \citet{jW05}]. Such a result has recently been rederived
[\citet{dV09}] in the context of approximating the minimal quasistationary
distribution of a diffusion generator with discrete spectrum via
interacting particle systems of Fleming--Viot type.
\end{remark}

%re3.6 #&#
\begin{remark}\label{heatasympt}
Assume that $\lambda_0^{\kappa}$ is an eigenvalue with associated
eigenfunction $u_{\lambda_0^{\kappa}} \in\mL^2$, which by general
theory is strictly positive and simple. Then
%
%e3.10 ###
%e3.6 #&#
\begin{equation} \label{E:heatasympt}
\lim_{t\rightarrow\infty}e^{\lambda_0^{\kappa}t}p^{\kappa}(t,x,y)
= c  u_{\lambda_0^{\kappa}}(x)u_{\lambda_0^{\kappa}}(y),
\end{equation}
where $c$ is a normalizing constant. This was proved in \citet{bS93}
for the transition function of Brownian motion on Riemannian manifolds
but the proof carries over without essential changes to our case.
\end{remark}

We will also make use of the following result which is a special case
of Theorem~3.1 in \citet{quasistat}.
%
%le3.7 #&#
\begin{lemma}[{[Theorem 3.1 in \citet{quasistat}]}]\label{qsdcomp}
Let $0 \leq f \in\mL^2$ with compact support $\operatorname{supp}(f) \subset
[0,\infty)$ be given, and let $\nu_f$ denote the measure $f(x)\rrho
(x)\,dx$. Let $L^{\kappa}$ be as in Lemma \ref{spectrum} and let
$p^{\kappa}(t,\cdot,\cdot)$ denote the integral kernel of
$e^{-tL^{\kappa}}$. Then for arbitrary measurable bounded sets $A, B
\subset(0,\infty)$
\[
\lim_{t \rightarrow\infty}\frac{\int_0^{\infty}f(x)\int
_Bp^{\kappa}(t,x,y) \rrho(y) \rrho(x)\,dy\,dx}{\int_0^{\infty
}f(x)\int_Ap^{\kappa}(t,x,y) \rrho(y)\rrho(x)\,dy\,dx} = \frac
{\int_B\varphi(\lambda_0^{\kappa},y) \rrho(y)\,dy}{\int_A\varphi
(\lambda_0^{\kappa},y) \rrho(y)\,dy},
\]
that is, $X_t$ converges from the initial distributions $\frac{\nu
_f}{\int_0^{\infty}f(s) \rrho(ds)}$ on compacta to the
quasistationary distribution $\varphi(\lambda_0^{\kappa},\cdot)$.
\end{lemma}

The above lemma can be proved directly using the spectral
representation for Sturm--Liouville operators. The reader will see the
necessary arguments later in this work in the proof of Theorem \ref
{generallocalMandl}. Our first goal is to extend this result to the
case of general compactly supported initial
distributions~$\nu$.\vadjust{\goodbreak}

We begin by deducing some consequences of Lemma \ref{qsdcomp}.
This will lead to Proposition \ref{ratiolimit}, which is a ``strong
ratio limit theorem.'' Before we start proving the strong ratio limit
theorem we explain another analytic fact which has no direct relation
to spectral theory but which will turn out to be very useful.

%s3.3 ###
%s3.3 #&#
\subsection{Local parabolic Harnack inequality} \label{sec:LPHI}
A crucial tool for smoothing analytic information about the transition
kernel between different times and sites is the local parabolic Harnack
inequality, which quite generally holds for second order parabolic
differential equations. One version appropriate to our current purposes
may be found in \citet{gL96}, and states that for fixed $x_0, t_0 \in
(0,\infty)$ and $R>0$ there is a constant $C$ such that for every weak
solution $u$ of $(\partial_t - L^{\kappa})u = 0$ which is
nonnegative in $Q((x_0,t_0),4R) \subset(0,\infty) \times(0,\infty)$,
\[
\sup_{\Theta((x_0,t_0),R/2)} u \leq C\inf_{Q((x_0,t_0),R)} u,
\]
where
\[
Q((x_0,t_0),R) = \bigl\lbrace(x,t) \in\mathbb{R}^2 |\max
\bigl(|x-x_0|,\sqrt{|t-t_0|} \bigr)<R,t<t_0\bigr\rbrace
\]
and $\Theta((x_0,t_0),R) = Q((x_0,t_0-R^2),R)$. As in Theorem 10 of
\citet{eD97} this inequality can be applied to the transition kernel
$p^{\kappa}(t,x,y)$ in order to prove that for every compact $K
\subset(0,\infty)$ and $T>0$ there is a constant $c=c(K,T)> 0$ such
that for $t\geq T$, $x_1,x_2,x_3,x_4 \in K$
%
%e3.11 ###
%e3.7 #&#
\begin{equation} \label{E:harnack}
c^{-1}p^{\kappa}(t,x_1,x_2) \leq p^{\kappa}(t,x_3,x_4) \leq
cp^{\kappa}(t,x_1,x_2).
\end{equation}
Moreover the local parabolic Harnack inequality shows that there exists
a~locally bounded function $\zeta\dvtx(0,\infty)\rightarrow(0,\infty)$
such that for every $t \geq1$, $y>0$, and $x,z>0$ satisfying
$|z-x|<\frac{1}{2}\wedge\frac{|x|}{4}$
%
%e3.12 ###
%e3.8 #&#
\begin{equation} \label{E:harnack2}
p^{\kappa}(t,x,y) \leq\zeta(x)p^{\kappa}(t+1,z,y).
\end{equation}

%s3.4 ###
%s3.4 #&#
\subsection{Strong ratio limit theorem and convergence on compacta}
\label{sec:SRLT}
%
%le3.8 #&#
\begin{lemma}\label{weakcomp}
For any fixed $x_0 \in(0, \infty)$ the family of functions
\[
 \biggl\lbrace[0,\infty) \times\mathbb{R}_+ \times\mathbb{R}_+\ni
(t,x,y)\mapsto\frac{p^{\kappa}(t+s,x,y)}{p^{\kappa}(s,x_0,x_0)}\Big|
s \geq1 \biggr\rbrace
\]
is relatively compact in the space $C((0, \infty)^2, \mathbb{R})$ of
real-valued continuous functions on $(0,\infty)^2$, endowed with the
vague topology.
\end{lemma}

\begin{pf}
Let $(s_n)_{n \in\mathbb{N}}$ be a sequence with $1 \leq s_n
\rightarrow\infty$, and set for $t\in[0,\infty)$, $x,y \in(0,
\infty)$
\[
r_n(t,x,y) = \frac{p^{\kappa}(t+s_n,x,y)}{p^{\kappa}(s_n,x_{0},x_{0})},\vadjust{\goodbreak}
\]
where $a \in(0,\infty)$ is fixed. The functions $(t,x,y) \mapsto
r_n(t,x,y)$ ($n\in\mathbb{N}$) are solutions to the parabolic
equation
\[
(2\partial_t + L^{\kappa}_x + L^{\kappa}_y)r_n(t,x,y)=0,
\]
where the operator $L_x^{\kappa}$ and $L_y^{\kappa}$ act as
$L^{\kappa}$ on $x$- and $y$-variable, respectively. By the local
parabolic Harnack inequality (see Section \ref{sec:LPHI}) we conclude
that for each compact set $K \subset(0,\infty)$ there exists a
constant $C_K$ such that for all $n\in\mathbb{N}$, $t \geq0$ and
$x,y,a \in K$
\[
p^{\kappa}(t+s_n,x,y) \leq C_K p^{\kappa}(t+s_n,x_{0},x_{0}).
\]
By general spectral theory it is proved in \citet{eD97} that $r
\mapsto p^{\kappa}(r,x_0,x_0)$ is nonincreasing. Therefore we
conclude that for $t\geq0$ and $x,y \in K$
\[
\frac{p^{\kappa}(t+s_n,x,y)}{p^{\kappa}(s_n,x_{0},x_{0})} \leq C_K.
\]
Theorem 6.28 in \citet{gL96}
%%% This isn't so obvious... Liebermann's book is pretty hard to dip
%into. Might there be a more accessible source?
shows that the set $\lbrace r_n \mid n \in\mathbb{N}\rbrace$ is
locally uniformly equicontinuous. Therefore by the theorem of
Arzela--Ascoli [\citet{jK55}, Theorem 17], there exists a subsequence
$(r_{n_k})_{k \in\mathbb{N}}$ which converges locally uniformly.
\end{pf}

The proof above is modeled on Theorem 2.2 of \citet{ABT02}. Since
that theorem assumed the operator was critical and the coefficients
were H\"{o}lder-continuous, some modification was required

The analytic core of quasilimiting behavior is the convergence of
ratios of transition kernels, which we state and prove here as
Proposition \ref{ratiolimit}. This will imply convergence to the
quasistationary distribution on compacta, Theorem \ref
{generallocalMandl}. Convergence on the whole state space will then
require a~consideration of the recurrence or transience, to decide
whether most of the mass stays in a compact interval or escapes to infinity.

Results comparing transition probabilities at different times and
sites, in the limit as time goes to infinity, are commonly referred to
as strong ratio limit theorems. Strong ratio limit theorems for certain
branching processes can be found in \citet{AN72}. A proof of the
strong ratio property for certain Markov chains on the integers was
given in \citet{hK95}.

%pr3.9 #&#
\begin{Prop}\label{ratiolimit}
For any $a \in(0, \infty)$
\[
\lim_{s\rightarrow\infty}\frac{p(t+s,x,y)}{p(s,a,a)} = e^{-\lkap
t}\frac{\varphi({\lambda^ {\kappa}_0},x)\varphi(\lambda^{\kappa
}_0,y)}{\varphi(\lambda^{\kappa}_0,a)\varphi(\lambda^{\kappa}_0,a)}.
\]
\end{Prop}

\begin{pf}
For every sequence $(s_n)_{n\in\mathbb{N}} \subset(0,\infty)$
converging to infinity we know by Lemma \ref{weakcomp} that for some
subsequence $(s_{n_k})_k$ of $(s_n)$ there exists a~function $\psi$
such that
\[
\frac{p^{\kappa}(t+s_{n_k},x,y)}{p^{\kappa}(s_{n_k},a,a)}
\rightarrow\psi(t,x,y),
\]
where the convergence is locally uniform in $[0,\infty)\times
(0,\infty)^2$. Since by Lem\-ma~7.5 in \citet{quasistat} [see also
\citet{eD97}] for every $f \in\mL^2$ with compact support
\[
\lim_{s\rightarrow\infty}\frac{\langle e^{-(t+s)L^{\kappa
}}f,f\rangle}{\langle e^{-sL^{\kappa}}f,f\rangle} = e^{-\lambda
_0^{\kappa}t},
\]
one easily concludes that
\[
\psi(t,x,y) = e^{-\lambda_0^{\kappa}(t)}\psi(0,x,y).
\]
Lemma \ref{qsdcomp} shows that for every $f,g, h\in C^{\infty
}_0(0,\infty)$
\begin{eqnarray*}
&&\frac{\int_0^{\infty} g(y)\varphi(\lambda_0^{\kappa},y) \rrho
(y)\,dy}{\int_0^{\infty}h(y)\varphi(\lambda_0^{\kappa},y) \rrho
(y)\,dy} \\
&&\qquad = \lim_{k \rightarrow\infty}\frac{\int_0^{\infty}f(x)\int
_0^{\infty}g(y)p^{\kappa}(s_{n_k},x,y) \rrho(y)\rrho(x) \, dy \,
dx}{\int_0^{\infty}f(x)\int_0^{\infty}h(y)p^{\kappa}(s_{n_k},x,y)
\rrho(y)\rrho(x) \, dy \, dx} \\
&&\qquad = \lim_{k \rightarrow\infty}\frac{\int_0^{\infty}f(x)\int
_0^{\infty}g(y)({p^{\kappa}(s_{n_k},x,y)}/{p^{\kappa
}(s_{n_k},x_0,x_0)}) \rrho(y)\rrho(x) \, dy \, dx}{\int_0^{\infty
}f(x)\int_0^{\infty}h(y)({p^{\kappa}(s_{n_k},x,y)}/{p^{\kappa
}(s_{n_k},x_0,x_0)}) \rrho(y)\rrho(x) \, dy \, dx} \\
&&\qquad = \frac{\int_0^{\infty}f(x)\int_0^{\infty}g(y) \psi(0,x,y)
\rrho(y)\rrho(x) \, dy \, dx}{\int_0^{\infty}f(x)\int_0^{\infty
}h(y) \psi(0,x,y) \rrho(y)\rrho(x) \, dy \, dx} .
\end{eqnarray*}
This implies that for $x \in(0,\infty)$, $g,h\in C^{\infty
}_0((0,\infty)$
\begin{eqnarray*}
&&\frac{\int_0^{\infty} g(y)\varphi(\lambda_0^{\kappa},y) \rrho
(y)\,dy}{\int_0^{\infty}h(y)\varphi(\lambda_0^{\kappa},y) \rrho
(y)\,dy} \int_0^{\infty}h(y) \psi(0,x,y) \rrho(y)\,dy\\
&&\qquad =\int_0^{\infty
}g(y) \psi(0,x,y) \rrho(y)\,dy,
\end{eqnarray*}
and hence for every $h$
\[
\psi(0,x,y) = \varphi(\lambda_0^{\kappa},y)\frac{\int_0^{\infty
}h(z) \psi(0,x,z) \rrho(z)\,dz}{\int_0^{\infty}h(z)\varphi(\lambda
_0^{\kappa},z) \rrho(z)\,dz}.
\]
Due to the symmetry of $\psi(0,\cdot,\cdot)$ we conclude that for
some constant $c \geq0$
\[
\psi(0,x,y) = c  \varphi(\lambda_0^{\kappa},x)\varphi(\lambda
_0^{\kappa},y).
\]
Because of $\psi(0,a,a) = 1$ we arrive at $c^{-1} = \varphi(\lambda
_0^{\kappa},a)\varphi(\lambda_0^{\kappa},a)$. Since this is true
for every subsequence, the assertion of the theorem is proved.
\end{pf}

%co3.10 #&#
\begin{Corollary}\label{C:ratiolimitintegrated}
If $\nu$ is any compactly supported initial distribution, and $f$ a
nonnegative compactly supported measurable function with $\nu[f]>0$,
then for any fixed $t$,
\[
 |\log\E_{\nu}[f(X_{t+s})] -\log\E_{\nu}[f(X_{s})]  |
\]
is bounded for $s\in\R^{+}$.
\end{Corollary}

\begin{pf}
Since $g(s):=\int\!\!\int p^{\kappa}(s,x,y)\rrho(y) f(y) \,d\nu(x)$ is
positive and continuous, it suffices to show that
\[
0<\liminf_{s\to\infty} \frac{g(t+s)}{g(s)}\le\limsup_{s\to\infty
} \frac{g(t+s)}{g(s)}<\infty.
\]
By \eqref{E:harnack} we may find positive $c$ such that for $s$
sufficiently large, and any $x_{*},y_{*}\in K:=\supp(\nu)\cup\supp(f)$,
\[
c^{-1} p^{\kappa}(s,x_{*},y_{*})\Rho[f]\le g(s)\le c p^{\kappa
}(s,x_{*},y_{*})\Rho[f].
\]
Thus,
\[
c^{-2} \frac{p^{\kappa}(s+t,x_{*},y_{*})}{p^{\kappa
}(s,x_{*},y_{*})}\le\frac{g(s+t)}{g(s)}\le c^{2} \frac{p^{\kappa
}(s+t,x_{*},y_{*})}{p^{\kappa}(s,x_{*},y_{*})}.
\]
By Proposition \ref{ratiolimit} the upper and lower bounds converge to
$c^{2}e^{-\lkap}$ and $c^{-2}e^{-\lkap}$, respectively, as $s\to
\infty$.
\end{pf}

%re3.11 #&#
\begin{remark}
\label{rem:Martin}
In terms of parabolic Martin boundary theory Proposition \ref
{ratiolimit} says that every sequence $(s_n,x) \subset(0,\infty)
\times(0,\infty)$ with $\lim_{n\rightarrow\infty}s_n = \infty$
converges in the parabolic Martin topology to the parabolic Martin
boundary point corresponding to the minimal parabolic function
$h_{\lambda_0^{\kappa}}(t,x) = e^{\lambda_0^{\kappa}t}\varphi
(\lambda_0^{\kappa},x)$. The parabolic function $h_{\lambda
_0^{\kappa}}$ must actually be invariant, since it corresponds to a
point in the parabolic Martin boundary whose time coordinate is $\infty$.
\end{remark}

%re3.12 #&#
\begin{remark}\label{Daviesconjecture}
The existence of strong ratio limits for general symmetric diffusion,
that is, the existence of
\[
\lim_{t\rightarrow\infty}\frac{p(t,x,y)}{p(t,x_0,y_0)},
\]
where $p(t,\cdot,\cdot)$ denotes the transition kernel of the
diffusion, was investigated under special conditions by \citet{eD97}
and is now often referred to as Davies's conjecture. In a private
communication, Gady Kozma disproved this conjecture by presenting a
counterexample. Proposition \ref{ratiolimit} shows that in one
dimension the Davies conjecture is true, if one boundary point is
regular. It is an open question whether the Davies conjecture
generally holds in one dimension.
\end{remark}

%pr3.13 #&#
\begin{Prop} \label{P:Harnackconstantlimit}
For all positive $z$, including $z=\infty$,
\[
\limsup_{t\to\infty} \frac{1}{t} \log\P_{x} \{ X_{t} \le
z \}  \quad \mbox{and}\quad  \liminf_{t\to\infty} \frac{1}{t}
\log\P_{x} \{ X_{t} \le z  \}
\]
are both constant in $x>0$. Hence also
\[
\limsup_{t\to\infty} \frac{1}{t} \log\P_{x} \{ X_{t} \le z
\cond\tp>t \}  \quad \mbox{and}\quad  \liminf_{t\to\infty}
\frac{1}{t} \log\P_{x} \{ X_{t} \le z \cond\tp>t \}
\]
are both constant in $x>0$.
\end{Prop}

\begin{pf}
We prove only the first statement for $\limsup$, the other proof being
identical. Suppose we have $x,x'$ such that
%
%e3.13 ###
%e3.9 #&#
\begin{equation} \label{E:xxprime}
\limsup_{t\to\infty} \frac{1}{t} \log\P_{x} \{ X_{t} \le z
 \} < \limsup_{t\to\infty} \frac{1}{t} \log\P_{x'} \{
X_{t} \le z  \} .
\end{equation}
We may assume without loss of generality that $|x-x'|< \frac
{1}{2}\wedge\frac{|x|}{4}\wedge\frac{|x'|}{4}$. (If not, then there
must be other starting points closer together where the limits differ.)
Applying \eqref{E:harnack2}, we have for all $t\geq1$,
\[
\P_{x'} \{ X_{t} \le z  \} \leq\zeta(x')\P_{x} \{
X_{t+1} \le z  \} ,
\]
so that
\begin{eqnarray*}
\limsup_{t\to\infty} \frac{1}{t}\log\P_{x'} \{ X_{t} \le
z \} &\leq&\limsup_{t\to\infty}\frac{1}{t}\log\P_{x} \{
X_{t} \le z  \}\\
&&{} +\limsup_{t\to\infty} \frac{1}{t} \log\frac{\P_{x}
\{ X_{t+1} \le z  \}}{\P_{x} \{ X_{t} \le z  \}}.
\end{eqnarray*}
The limits on the second line are 0 by Corrolary \ref
{C:ratiolimitintegrated}. We have then a~contradiction to \eqref
{E:xxprime}, which completes the proof.
\end{pf}

%s3.5 ###
%s3.5 #&#
\subsection{Conditions on the initial distribution} \label{sec:ID}
In their version of the ratio limit theorem [\citeauthor{quasistat} (\citeyear{quasistat}), Theorem
3.1], the authors had to pose an additional condition on the
initial distribution $\nu$, and they stated the general case as an
open problem. Their most general condition reads\looseness=1
{\renewcommand{\theequation}{$\mathit{ID}'$}
%e3.10 #&#
\begin{equation}\label{eqID}
\qquad \begin{tabular}{@{}p{300pt}@{}}
If $X_0$ has distribution $\nu$, then
$\exists s \geq0$ for which the distribution of~$X_s$ has a density
$f \in\mL^2$, with $\liminf_{\lambda
\downarrow\lambda_0^{\kappa}} Uf(\lambda) > - \infty$,
\end{tabular}
\end{equation}}

\vspace*{-8pt}
\setcounter{equation}{8}\looseness=0

\noindent where the $U$ denotes the unitary operator from Theorem \ref
{weylspectraltheorem}. [In \citet{quasistat} the definition of the
operator $U$ is slightly different. This is connected to the fact that
there the authors work $L^2$ spaces with respect to the Lebesgue
measure.] This condition is obviously satisfied for compactly supported
initial distributions having a density, but it is not obvious how to
verify that a given general initial distribution $\nu$ with compact
support satisfies the condition (\ref{eqID}). Using some results from
spectral theory and several ideas from \citet{quasistat}, we can
remove this restriction. An essential ingredient in the proof is Lemma
\ref{L:sqrtcompare}, which allows us to ignore the upper end of the
spectrum for large $t$.

%le3.14 #&#
\begin{lemma}\label{L:lowspec}
Given $g\in\mL^{2}$, we have
%
%e3.14 ###
%e3.9 #&#
\begin{equation} \label{E:lowspec}
\lim_{t\to\infty} t^{-1}\log \| e^{-tL^{\kappa}}g  \|
=-\lambda_{g}.
\end{equation}
\end{lemma}

\begin{pf}
By the spectral theorem \eqref{spectraltheorem3} we know that
\begin{eqnarray*}
\limsup_{t\to\infty} t^{-1}\log \| e^{-tL^{\kappa}}g  \|^{2}&=&\limsup_{t\to\infty} t^{-1}\log\int_{\lambda_{g}}^{\infty
} e^{-2t\lambda} \,d \| E^{\kappa}_{\lambda} g \|^{2}(\lambda)\\
&\le&-2\lambda_{g} +\limsup_{t\to\infty} t^{-1}\log \| g \|^{2}\\
&=& -2\lambda_{g}.
\end{eqnarray*}
For the lower bound we take any $\lambda_{*}>\lambda_{g}$, and have
\begin{eqnarray*}
\liminf_{t\to\infty} t^{-1}\log \| e^{-tL^{\kappa}}g  \|
^{2}&\ge&\liminf_{t\to\infty} t^{-1}\log\int_{\lambda
_{g}}^{\lambda_{*}} e^{-2t\lambda} \,d \| E^{\kappa}_{\lambda}
g \|^{2} (\lambda)\\
&\ge&-2\lambda_{*} +\liminf_{t\to\infty} t^{-1}\log \|
E^{\kappa}  ([\lambda_{g},\lambda_{*}] )g \|^{2}.
\end{eqnarray*}
Since $\lambda_{g}$ is in the support of $d\|E^{\kappa}g\|$, this is
equal to $-2\lambda_{*}$. Since this is true for any $\lambda
_{*}>\lambda_{g}$, this completes the proof of \eqref{E:lowspec}.
\end{pf}

For a Radon measure $\nu$ on $(0,\infty)$ and a Borel measurable
function $f\dvtx\break(0,\infty)\mapsto\mathbb{C}$ we use the notation
$\langle\nu, f\rangle:= \int_0^{\infty} f(s) \nu(ds)$.
%
%th3.15 #&#
\begin{theorem}\label{generallocalMandl}
The killed diffusion $X_t$ converges on compacta to the quasistationary
distribution with density proportional to $\varphi(\lambda_0^{\kappa
},\cdot)$ from any initial distribution which is compactly supported
in $(0,\infty)$.
\end{theorem}

\begin{pf}
An application of Weyl's eigenfunction expansion theorem and Fubini's
theorem tells us that the operator $E^\kappa([\lambda_0^{\kappa
},\lambda_1))e^{-tL^{\kappa}}$ has a continuous integral kernel
%
%e3.15 ###
%e3.10 #&#
\begin{equation}\label{truncatedintegralkernel}
(t,x,y) \mapsto h^{\lambda_1}(t,x,y) = \int_{[\lambda_0^{\kappa
},\lambda_1]}e^{-t\lambda}\varphi(\lambda,x)\varphi(\lambda,y)\,
d\sigma(\lambda)
\end{equation}
with respect to the measure $\Rho$. This implies that for every
compact subset $K \subset[0,\infty)$ the function $E^\kappa([\lambda
_0^{\kappa},\lambda_1])e^{-tL^{\kappa}}\mathbf{1}_K$ is continuous
and therefore
\[
\langle\nu,E^\kappa([\lambda_0^{\kappa},\lambda_1])e^{-tL^{\kappa
}}\mathbf{1}_K\rangle= \int_{\mathbb{R}}E^\kappa([\lambda
_0^{\kappa},\lambda_1])e^{-tL^{\kappa}}\mathbf{1}_K(x) \,d\nu(x)
\]
is well defined. For every Borel set $A \subset[0,z]$, then
%
%e3.16 ###
%e3.11 #&#
\begin{eqnarray}\label{kernelform}\qquad
&&\nu [ E^\kappa([\lambda_0^{\kappa},\lambda_1])e^{-tL^{\kappa
}}\mathbf{1}_A ] \nonumber\\[-8pt]\\[-8pt]
&&\qquad = \int_{[\lambda_0^{\kappa},\lambda
_1]}e^{-t\lambda} \biggl[\int_{0}^{z}\varphi(\lambda,x) \, d\nu
(x)\int_A\varphi(\lambda,y)\,d\Rho(y)  \biggr]\,d\sigma(\lambda).\nonumber
\end{eqnarray}

Let $g\dvtx[\lambda_{0}^{\kappa},\infty)\to\R$ be any continuous
function. Then
%
%e3.17 ###
%e3.12 #&#
\begin{equation} \label{gbound}\qquad
\limsup_{t \rightarrow\infty}  \biggl|\frac{\int_{[\lambda
_0^{\kappa},\lambda_1]}e^{-t\lambda}g(\lambda)\,d\sigma(\lambda
)}{\int_{[\lambda_0^{\kappa},\lambda_1]}e^{-t\lambda}\,d\sigma
(\lambda)}-g(\lambda_0^{\kappa})\biggr | \le\sup_{[\lambda
_0^{\kappa},\lambda_1]}|g(\lambda)-g(\lambda_0^{\kappa})|.
\end{equation}
As in the proof of Theorem 3.1 in \citet{quasistat}, for any $\lambda
_1,\tilde{\lambda}_1,\lambda_2 > \lambda_0^{\kappa}$ set $\lambda
_{*}=\lambda_1\wedge\tilde{\lambda}_1\wedge\lambda_2\wedge\tilde
\lambda_2$ and $\lambda^{*}=\lambda_1\vee\tilde{\lambda}_1\vee
\lambda_2\vee\tilde\lambda_2$. Then we have the bound
\begin{eqnarray*}
&& \biggl|\frac{\int_{[\lambda_0^{\kappa},\lambda_1]}e^{-t\lambda
}g(\lambda)\,d\sigma(\lambda)}{\int_{[\lambda_0^{\kappa},\lambda
_2]}e^{-t\lambda}\,d\sigma(\lambda)}-\frac{\int_{[\lambda
_0^{\kappa},\tilde{\lambda}_1]}e^{-t\lambda}g(\lambda)\,d\sigma
(\lambda)}{\int_{[\lambda_0^{\kappa},\tilde\lambda
_2]}e^{-t\lambda}\,d\sigma(\lambda)} \biggr|\\[-2pt]
&&\qquad \leq e^{(\lambda_0^{\kappa}-\lambda_*)t}\frac{\int_{[\lambda_0^{\kappa},\lambda
^{*}]}|g(\lambda)|\,d\sigma(\lambda)}{\int_{[\lambda_0^{\kappa
},\lambda_{*})]}\,d\sigma(\lambda)},
\end{eqnarray*}
which tells us that
%
%e3.18 ###
%e3.13 #&#
\begin{equation}\label{lambdaindependence}
\limsup_{t \rightarrow\infty}\frac{\int_{[\lambda_0^{\kappa
},\lambda_1]}e^{-t\lambda}g(\lambda)\,d\sigma(\lambda)}{\int
_{[\lambda_0^{\kappa},\lambda_2]}e^{-t\lambda}\,d\sigma(\lambda)}
 \mbox{ is independent of }\lambda_1, \lambda_2 \in(\lambda
_0^{\kappa},\infty).
\end{equation}
Since $g$ is continuous, \eqref{gbound} and \eqref
{lambdaindependence} combine to show that for a general positive
continuous $h$,
%
%e3.19 ###
%e3.14 #&#
\begin{equation}\label{gundeins}
\lim_{t \rightarrow\infty}  \biggl|\frac{\int_{[\lambda_0^{\kappa
},\lambda_1]}e^{-t\lambda}h(\lambda)g(\lambda)\,d\sigma(\lambda
)}{\int_{[\lambda_0^{\kappa},\lambda_1]}h(\lambda)e^{-t\lambda}\,
d\sigma(\lambda)}-g(\lambda_0^{\kappa}) \biggr| = 0.
\end{equation}
By \eqref{kernelform} we now see that for every $\lambda_1 \in
(\lambda_0^{\kappa},\infty)$
%
%e3.20 ###
%e3.15 #&#
\begin{eqnarray}\label{lokalspektralconv}
&&\lim_{t \rightarrow\infty} \frac{\nu[ E^\kappa([\lambda
_0^{\kappa},\lambda_1])e^{-tL^{\kappa}}\mathbf{1}_A]}{\nu[
E^\kappa([\lambda_0^{\kappa},\lambda_1])e^{-tL^{\kappa}}\mathbf
{1}_{[0,z]}]} \nonumber\\[-2pt]
&&\qquad  = \lim_{t \rightarrow\infty}
\frac{\int_{[\lambda_0^{\kappa},\lambda_1]}e^{-t\lambda}
[\int_{0}^{z}\varphi(\lambda,x) \, d\nu(x)\int_A\varphi(\lambda
,y)\,d\Rho(y)  ]\,d\sigma(\lambda)}
{\int_{[\lambda_0^{\kappa},\lambda_1]}e^{-t\lambda} [\int
_{0}^{z}\varphi(\lambda,x) \, d\nu(x)\int_0^{z}\varphi(\lambda
,y)\,d\Rho(y) ]\,d\sigma(\lambda)}\\[-2pt]
&&\qquad  =\frac{\int_A\varphi(\lambda_0^{\kappa},y) \rrho
(y)\,dy}{\int_0^z\varphi(\lambda_0^{\kappa},y) \rrho(y)\,dy}.\nonumber
\end{eqnarray}

The assertion of the theorem follows immediately from \eqref{lokalspektralconv} once it is shown that
%
%e3.21 ###
%e3.16 #&#
\begin{equation}\label{finalconvcomp}
\lim_{t \rightarrow\infty}\frac{\mathbb{P}_{\nu}(X_t \in
A)}{\mathbb{P}_{\nu}(X_t \leq z)} = \lim_{t \rightarrow\infty
}\frac{ \nu [ E^\kappa([\lambda_0^{\kappa},\lambda
_1])e^{-tL^{\kappa}}\mathbf{1}_A ]}{\nu [ E^\kappa
([\lambda_0^{\kappa},\lambda_1])e^{-tL^{\kappa}}\mathbf
{1}_{[0,z]} ]}.
\end{equation}
Observe that
%
%e3.17 #&#
\begin{eqnarray}\label{laststep}
&&\frac{\nu [ e^{-tL^{\kappa}}\mathbf{1}_A ]}{\nu [
E^\kappa([0,\lambda_1])e^{-tL^{\kappa}}\mathbf{1}_A ]}\nonumber  \\[-2pt]
&&\qquad =
\frac{\nu [ E^\kappa([0,\lambda_1])e^{-tL^{\kappa}}\mathbf
{1}_A ]+\nu [ E^\kappa((\lambda_1,\infty))e^{-tL^{\kappa
}}\mathbf{1}_A ]}{\nu [ E^\kappa([0,\lambda
_1])e^{-tL^{\kappa}}\mathbf{1}_A ]}\\[-2pt]
&&\qquad = 1 +\frac{\nu [ E^\kappa((\lambda_1,\infty))e^{-tL^{\kappa
}}\mathbf{1}_A ]}{\nu [ E^\kappa([0,\lambda
_1])e^{-tL^{\kappa}}\mathbf{1}_A ]}.\nonumber
\end{eqnarray}
Since $ e^{-tL^{\kappa}}\mathbf{1}_A$ and $E^\kappa([0,\lambda
_1])e^{-tL^{\kappa}}\mathbf{1}_A$ are continuous, the function
$E^\kappa((\lambda_1,\break\infty))e^{-tL^{\kappa}}\mathbf{1}_A$ must
also be continuous.\vadjust{\goodbreak} Thus $\nu[ E^\kappa((\lambda_1,\infty
))e^{-tL^{\kappa}}\mathbf{1}_A]$ is well defined. By Lemma \ref
{L:sqrtcompare},
%
%e3.22 ###
%e3.18 #&#
\begin{equation}\label{Zaehler}
-\lim_{t \rightarrow\infty}\frac{1}{t}\log |\langle\nu,
E^\kappa ((\lambda_1,\infty) )e^{-tL^{\kappa}}\mathbf
{1}_A\rangle | \geq\lambda_1.
\end{equation}
As $\varphi(\lambda^{\kappa}_0,x) > 0$ for every $x \in(0,\infty)$
there is, by continuity, $\lambda_1 > \lambda^{\kappa}_0$ such that
for every $\lambda\in[\lambda^{\kappa}_0,\lambda_1]$
\[
\int_0^{\infty}\varphi(\lambda,x)\,d\nu(x) \mbox{ and } \int
_A\varphi(\lambda,y)\rrho(y)\,dy \mbox{ are both positive}.
\]
Then it is easy to see that
%
%e3.23 ###
%e3.19 #&#
\begin{equation}\label{Nenner}
 -\lim_{t \rightarrow\infty}\frac{1}{t}\log\!\int_{[\lambda^{\kappa
}_0,\lambda_1]}e^{-\lambda t}\!\int_0^{\infty}\!\varphi(\lambda,x)\,d\nu
(x)\!\int_A\varphi(\lambda,y)  \rrho(y)\,dy\,d\sigma(\lambda) \leq
\lambda_0^{\kappa}.\hspace*{-35pt}
\end{equation}
Equations \eqref{laststep}, \eqref{Zaehler} and \eqref{Nenner}
combine to show that\vspace*{-1pt}
\[
\lim_{t \rightarrow\infty}\frac{\nu [ e^{-tL^{\kappa
}}\mathbf{1}_A ]}{\nu [E^\kappa([0,\lambda
_1])e^{-tL^{\kappa}}\mathbf{1}_A ]} = 1,
\]
and therefore \eqref{finalconvcomp}.
\end{pf}

%s3.6 ###
%s3.6 #&#
\subsection{\texorpdfstring{Entrance boundary at $\infty$}{Entrance boundary at infinity}} \label{sec:entrance}
As mentioned above, with the exception of the recent work [\citet
{6authors}], work on these problems has generally assumed that $0$ is
regular and $\infty$ natural. Intuitively, we should expect the
problems to be easier if $\infty$ is an entrance boundary. We show
that this is indeed the case in Theorem \ref{nichtnatuerlich}, as the
spectrum of the operator $L^{\kappa}$ is purely discrete. This
interesting fact has not been mentioned by previous authors [cf.
Section 3 in \citet{6authors}] working on quasistationary distributions
for one-dimensional diffusions. The proof relies on standard ideas from
the spectral theory of differential operators.\vspace*{-2pt}

%th3.16 #&#
\begin{theorem}\label{nichtnatuerlich}
If $\infty$ is an entrance boundary, then the spectrum of $L^{\kappa
}$ is discrete.\vspace*{-2pt}
\end{theorem}

\begin{pf}
Assume that $0$ is a regular boundary point, and we begin by
considering the case $\kappa\equiv0$. Let $f$ be a solution to the
eigenvalue equation $-\frac{1}{2\rrho} (\rrho f')'=\lambda f$ on
$(0,\infty)$, for some $\lambda>0$.

Let $x>1$ be any local maximum, and $\tilde{x}>x$ the first local
minimum following $x$ (assuming there is one). Since $f$ solves the
equation $\tau u = \lambda u$ one easily sees that local maxima are
positive and local minima negative. Integrating by parts and using the
fact that $f'(x)=f'(\tx)=0$, we have
\begin{eqnarray*}
0&<&f(x) - f(\tx) = \int_{\tx}^{x}f'(y)\,dy\\
&=&\int_{\tx}^{x}(\rrho(y)f'(y)) \rrho(y)^{-1}\,dy\\
&=&\int_{x}^{\tx}(\rrho(y)f'(y))' \int_{1}^{y}\rrho(z)^{-1}\,dz \,dy\\
&=&-2\lambda\int_{x}^{\tx}\rrho(y)f(y) \int_{1}^{y}\rrho(z)^{-1}\,dz \,dy\\
&<&2\lambda \bigl(f(x)-f(\tx)  \bigr)\int_{x}^{\tx}\rrho(y) \int_{1}^{y}\rrho(z)^{-1}\,dz \,dy,
\end{eqnarray*}
from which we conclude that
\[
\frac{1}{2\lambda}<\int_{x}^{\tx}\rrho(y) \int_{1}^{y}\rrho
(z)^{-1}\,dz \,dy.
\]
Since we have assumed that $\infty$ is an entrance boundary, we know that
\begin{eqnarray*}
\infty&>&\int_{1}^{\infty}\rrho(y) \int_{1}^{y}\rrho(z)^{-1}\,dz \,dy\\
&\ge&\sum_{\mathrm{pairs }\ (x,\tx)}\int_{x}^{\tx}\rrho(y) \int
_{1}^{y}\rrho(z)^{-1}\,dz \,dy\\
&\ge&(2\lambda)^{-1} \cdot\#\mbox{pairs }(x,\tx).
\end{eqnarray*}
Since the zeroes of $f$ are separated by alternating local minima and
maxima, it follows that $f$ has only finitely many zeroes on $(1,\infty
)$, hence only finitely many zeroes in all. It follows from a theorem
of Hartmann [\citet{jW67}, Theorem~1.1] that the spectrum of $L^{0}$ (the
operator with $\kappa\equiv0$) is discrete.

Suppose now that the spectrum of $L^{\kappa}$ is not discrete. Then
there is a $\lambda_{*}$ such that $E_{\lambda_{*}}$ has
infinite-dimensional range. All such $v \in\operatorname
{Ran}(E_{\lambda_{*}})$ are in the domain of $q^{\kappa}$ and satisfy
$q^{\kappa}(v,v)\le\lambda_{*}\|v\|^{2}$. But then they are also in
the domain of $q^{0}$ and satisfy $q^{0}(v,v)\le\lambda_{*}\|v\|
^{2}$. By the minimax principle for the discrete spectrum [cf.
\citeauthor{jW00} (\citeyear{jW00}), Theorem 8.8],  this contradicts the fact that $L$ has discrete spectrum.
\end{pf}

%re3.17 #&#
\begin{remark}\label{logistic}
There are general necessary and sufficient conditions for the
discreteness of the spectrum of Sturm--Liouville operators obtained in
\citet{CR02}, of which Theorem \ref{nichtnatuerlich} may be seen as
a special case. However, general versions found in the literature, such
as the main result in \citet{CR02} and Theorem 1 in \citet{rP09},
do not seem to be immediately applicable.
\end{remark}

%s4 ###
%s4 #&#
\section{Convergence to quasistationarity} \label{sec:qstat}
In this section we consider the problem of convergence to the Yaglom
limit. More precisely we ask for conditions, which ensure that $X_t$
converges to the quasistationary distribution given by the density
$\varphi(\lambda_0^{\kappa},\cdot)$. Recall that we always assume
that $0$ is regular.

%s4.1 ###
%s4.1 #&#
\subsection{The asymptotic measure and asymptotic killing rate} \label
{sec:asympt}
We begin by collecting some basic results about the asymptotic
distribution of the process on sets which may not be compact. These
results hold whenever $\lkap\ne K$, but will be required primarily in
Section \ref{sec:lowkill}, where $\lkap>K$.

As in \citet{quasistat} we define, for Borel sets $A$, the family of measures
\[
F_t(\nu,A) = \mathbb{P}_{\nu} (X_t \in A |\tau_{\partial}
> t  )
\]
and
\[
a_t(\nu,r) = \mathbb{P}_{\nu} ( \tau_{\partial} > t+r |
\tau_{\partial}> t ) = \int F_t(\nu,dy)\mathbb{P}_y (
\tau_{\partial} > r ).
\]
If the process $X_t$ started from the compactly supported initial
distribution~$\nu$ escapes to infinity, then for any sequence
$(t_n)_{n \in\mathbb{N}}$ converging to infinity the measures
$F_{t_n}(\nu,\cdot)$ converge weakly to point the measure $\delta
_{\infty}$. If the process $X_t$ started from $\nu$ converges to the
quasistationary distribution $\varphi$ then then the limit of
$F_{t_n}(\nu,dy)$ is concentrated on $\mathbb{R}_+$, and has the
density $\frac{\varphi(\lambda_0^{\kappa},\cdot)}{\int_0^{\infty
}\varphi(\lambda_0^{\kappa},y) \rrho(y)\,dy}$ with respect to $\Rho$.
The next lemma is in essence a combination of Lemma~5.3 and
Theorem~3.3 in \citet{quasistat}, together with our Lemma~\ref{qsdcomp}.

%le4.1 #&#
\begin{lemma}\label{SEcomb}
Assume that $\infty$ is a natural boundary point and suppose that
$\lambda^{\kappa}_0 \ne K$. Then the limit $a(\nu,r) = \lim
_{t\rightarrow\infty}a_{t}(\nu,r)$ exists, and satisfies
%
%e4.1 ###
%e4.1 #&#
\begin{eqnarray}\label{formelfuera}
a(\nu,r) &=& F(\nu,\mathbb{R}_+) \int\varphi(\lambda_0^{\kappa
},y)\mathbb{P}_y (\tau_{\partial}>r ) \rrho(y)\,dy \nonumber\\[-8pt]\\[-8pt]
&&{}+\bigl(1-F(\nu,\mathbb{R}_+)  \bigr) e^{-Kr}.\nonumber
\end{eqnarray}
Either $F(\nu,\mathbb{R}_+) = 0$ for every compactly supported
initial distribution $\nu$ or $F(\nu,\mathbb{R}_+) = 1$ for every
such $\nu$.

There exists $\eta_{\nu} \in\mathbb{R}$ (called the \textup
{asymptotic mortality rate}) such that
%
%e4.2 ###
%e4.2 #&#
\begin{equation}\label{eta}
a(\nu,r) = e^{-\eta_{\nu} r}.
\end{equation}
If the process escapes to infinity then $\eta_{\nu}=K$.
\end{lemma}

\begin{pf}
Let $\nu$ be a compactly supported initial distribution. Let $(t_n)
\subset(0,\infty)$ be a sequence converging to infinity. On the
compactification $[0,\infty]$ of $(0, \infty)$ the sequence of
measures $F_{t_n}(\nu,dy)$ has a limit point. By Theorem~\ref{Thm33}
this limit point is either a measure on $(0,\infty)$ which has the
density $\frac{\varphi(\lambda_0^{\kappa},\cdot)}{\int_0^{\infty
}\varphi(\lambda_0^{\kappa},y) \rrho(y)\,dy}$\vspace*{1pt} with respect to the
measure $\Rho$ or is the point mass at $\infty$. Theorem \ref{Thm33}
shows that there is only one limit point, and that the limit point is
independent of the sequence $(t_n)$ and the initial distribution $\nu$.
Thus $F_t(\nu,dy)$ converges weakly. If $\infty$ is natural, then
according to Proposition~3.1 in combination with Proposition 4.3 in
\citet{A74} the unkilled diffusion process ($\kappa= 0$) satisfies
$\lim_{y \rightarrow\infty}\mathbb{P}_{y}(T_a \leq t)=0$ for every
$a>0$. Due to our assumption on $\kappa$ we are given $\varepsilon>
0$ and sufficiently large $a>0$, and we therefore get ($x \geq a$)
\begin{eqnarray*}
e^{-(K-\varepsilon)t} &\leq&\mathbb{P}_x(\tau_{\partial}>t) =
\mathbb{E}_x \bigl[e^{-\int_0^t \kappa(X_s)\,ds};T_a\leq t \bigr]+
\mathbb{E}_x \bigl[e^{-\int_0^t \kappa(X_s)\,ds};T_a > t \bigr]\\
&\leq&\varepsilon+ e^{-(K+\varepsilon)t}.
\end{eqnarray*}
Thus we conclude that
\[
\lim_{y \rightarrow\infty}\mathbb{P}_y (\tau_{\partial}>
r ) = e^{-Kr}.
\]
This shows that
\begin{eqnarray*}
\lim_{t \rightarrow\infty}\int F_t(\nu,dy)\mathbb{P}_y ( \tau
_{\partial} > r ) &=& F(\nu,\mathbb{R}_+) \int\varphi(\lambda
_0^{\kappa},y)\mathbb{P}_y (\tau_{\partial}>r ) \rrho
(y)\,dy \\
&&{} +  \bigl(1-F(\nu,\mathbb{R}_+) e^{-Kr} \bigr),
\end{eqnarray*}
which is \eqref{formelfuera}.

For any $r,s\ge0$ we have
\begin{eqnarray*}
a(\nu,r+s) &=& \lim_{t \rightarrow\infty} \frac{\mathbb{P}_{\nu
}(\tau_{\partial} > t+r+s)}{\mathbb{P}(\tau_{\partial}>t)} \\
&=& \lim_{t \rightarrow\infty} \frac{\mathbb{P}_{\nu}(\tau
_{\partial} > t+r+s)}{\mathbb{P}_{\nu}(\tau_{\partial} > t+s)}
\frac{\mathbb{P}_{\nu}(\tau_{\partial}>t+s)}{\mathbb{P}_{\nu
}(\tau_{\partial}>t)}\\
&=&  \biggl(\lim_{t \rightarrow\infty} \frac{\mathbb{P}_{\nu}(\tau
_{\partial} > t+r+s)}{\mathbb{P}_{\nu}(\tau_{\partial} >
t+s)} \biggr)    \biggl(\lim_{t \rightarrow\infty} \frac{\mathbb
{P}_{\nu}(\tau_{\partial}>t+s)}{\mathbb{P}_{\nu}(\tau_{\partial
}>t)} \biggr)\\
&=& a(\nu,r)a(\nu,s),
\end{eqnarray*}
which directly implies \eqref{eta}. The final statement follows
directly from \eqref{formelfuera}.
\end{pf}

The quantity $\eta_{\nu}$ plays an important role. The reason for
this consists of the implication
%
%e4.3 ###
%e4.3 #&#
\begin{equation}\label{elementaryfact}
\lim_{t \rightarrow\infty}\frac{\mathbb{P}_{\nu} (\tau
_{\partial}>t+r )}{\mathbb{P}_{\nu} (\tau_{\partial
}>t )}=e^{-\eta_{\nu}r} \quad \Rightarrow\quad -\lim_{t\rightarrow\infty
}\mathbb{P}_{\nu} (\tau_{\partial}>t ) = \eta_{\nu},
\end{equation}
whose elementary proof is left to the reader. Thus in order to decide,
whether $X_t$ converges to the quasistationary distribution, we
investigate the asymptotic behavior of the function $r \mapsto\mathbb
{P}_{\nu} (\tau_{\partial} > r ) $, as $r \rightarrow
\infty$.

%le4.2 #&#
\begin{lemma}\label{exporder}
Suppose that $\Rho$ is a finite measure. Then for any compactly
supported initial distribution $\nu$ we have
\[
-\lim_{t\rightarrow\infty}\frac{1}{t}\log\mathbb{P}_{\nu}
(\tau_{\partial} > t ) = \lambda_0^{\kappa}.\vadjust{\goodbreak}
\]
\end{lemma}

\begin{pf}
Since $\Rho$ is finite, the constant function $\mathbf{1}$ is in $\mL
^2$. Lemma \ref{L:sqrtcompare} implies
%
%e4.4 #&#
\begin{eqnarray}
\limsup_{t \rightarrow\infty}\frac{1}{t}\log\mathbb{P}_{\nu
} (\tau_{\partial}> t ) &=&\limsup_{t \rightarrow\infty
}\frac{1}{t}\log\int
e^{-tL^{\kappa}}\indic(x)\,d\nu(x)\nonumber\\[-8pt]\\[-8pt]
&\leq&- \lambda_0^{\kappa}.\nonumber
\end{eqnarray}

Now we need a corresponding lower bound. Fix any $z>0$, and let $\mc
{I}:=\{x\dvtx|x-z|<\frac{1}{2}\wedge\frac{z}{4}\}$.
By \eqref{E:harnack2} there is a constant $C(z)=\sup_{w\in\mc{I}}
\zeta(w)$ such that
%
%e4.5 #&#
\begin{eqnarray}\label{lowerbound}
 \| e^{-{t}/{2}L^{\kappa}}\indic_{\mc{I}}  \|^{2}&=&
 \langle e^{-tL^{\kappa}}\indic_{\mc{I}} , \indic_{\mc{I}}
 \rangle\nonumber \\
&=&\int_{\mc{I}} \int_{\mc{I}} p^{\kappa}(t,x,y)\,d\Rho(y)\,d\Rho
(x)\nonumber\\[-8pt]\\[-8pt]
&\leq& C(z) \int_{\mc{I}} p^{\kappa}(t+1,z,y) \rrho(y)\,dy \nonumber \\
&\le& C(z) \mathbb{P}_z  (\tau_{\partial}>t+1 ).\nonumber
\end{eqnarray}
By Lemma \ref{L:lowspec} and Lemma \ref{L:supportbase} (using that
$\mathbf{1}_{\mc{I}}$ is strictly positive on $\mc{I}$) we see that
\[
\liminf_{t \rightarrow\infty}\frac{1}{t}\log\mathbb{P}_{z}
(\tau_{\partial}> t ) \ge-\inf\supp d\|E^{\kappa}\mathbf
{1}_\mc{I}\|^2(\lambda)=-\lambda_{0}^{\kappa}.
\]
Since by Proposition \ref{P:Harnackconstantlimit} the exponential rate
of decay of $\mathbb{P}_z  (\tau_{\partial}>t )$ is
locally constant in $z$ this completes the proof.
\end{pf}

%s4.2 ###
%s4.2 #&#
\subsection{\texorpdfstring{High killing at $\infty$}{High killing at infinity}}
In this section we consider the case where the asymptotic killing rate
is strictly bigger than $\lambda^{\kappa}_0$. Theorem \ref{Thm33}
shows that one has convergence to the quasistationary distribution if
and only if the lowest eigenfunction is integrable. We give a proof of
this assertion and moreover prove that the lowest eigenfunction is
actually always integrable. Therefore $\liminf\kappa>\lambda^{\kappa
}_0$ always implies convergence to the quasistationary distribution. In
contrast to \citet{quasistat}, we do not need to assume that $\infty
$ is a natural boundary. Thus $\infty$ is only assumed to be
inaccessible. Since according to Lemma \ref{spectrum}\hyperlink{it:isolated}{(v)} the bottom of the spectrum
is an isolated eigenvalue, the corresponding eigenfunction is
square-integrable, as well as $\lambda_0^{\kappa}$-invariant.

%th4.3 #&#
\begin{theorem}\label{limkbiggerev}
Suppose that $\liminf_{x\rightarrow\infty}\kappa(x) > \lambda
_0^{\kappa}$. Then we have for every Borel set $U \subset(0,\infty)$
%
%e4.4 ###
%e4.6 #&#
\begin{equation} \label{E:recurrent}
\lim_{t\rightarrow\infty}e^{\lambda^{\kappa}_0 t}\mathbb{P}^x(X_t
\in U; \tau_{\partial} > t) = u_{\lambda^{\kappa}_0}(x)\int_U
u_{\lambda^{\kappa}_0}(y) \rrho(y)\,dy,
\end{equation}
where $u_{\lambda_0^{\kappa}} \in\mL^2$ denotes the uniquely
determined (up to positive multiples) eigenfunction associated\vadjust{\goodbreak} to the
eigenvalue $\lambda_0^{\kappa}$. Moreover, the process $(X_t)$
associated to the Dirichlet form $q^{\kappa}$ converges to the
quasistationary distribution $u_{\lambda_0^{\kappa}}$.
\end{theorem}

The theorem will be the direct consequence of two lemmas: Lemma \ref
{L:L1L2}, which states that quasilimiting convergence follows whenever
the eigenfunction $u_{\lkap}$ is in $\mL^{2}$ and $\mL^{1}$; and
Lemma \ref{integrability}, which states that $u_{\lkap}$ is indeed
in~$\mL^{1}$ when $\liminf_{x\to\infty}\kappa(x)>\lkap$.

%le4.4 #&#
\begin{lemma}\label{L:L1L2}
Suppose $u_{\lkap}\in\mL^{1}\cap\mL^{2}$. Then \eqref
{E:recurrent} holds, and $X_t$ converges to the quasistationary
distribution $u_{\lkap}(y)\Rho(dy)/\int_0^{\infty}u_{\lkap}(y)\Rho(dy)$.
\end{lemma}

\begin{pf}
We know from Lemma \ref{spectrum} [part \hyperlink{it:isolated}{(v)} that
$\lambda^{\kappa}_0$ is an isolated eigenvalue. Therefore, the
eigenfunction $u_{\lambda^{\kappa}_0}$ is square integrable and satisfies
%
%e4.5 ###
%e4.7 #&#
\begin{equation} \label{E:L2ae}
 e^{-tL^{\kappa}}u_{\lambda^{\kappa}_0} = e^{-t\lambda^{\kappa
}_0}u_{\lambda^{\kappa}_0} \mbox{ in }\mL^{2},\mbox{ hence
identically (since $u_{\lkap}$ is continuous)}.\hspace*{-35pt}
\end{equation}
By \eqref{E:harnack2}, for $r>0$ sufficiently small,
\begin{eqnarray*}
p^{\kappa}(t,x,y) &=& \frac{\int_{B_r(x)}p^{\kappa}(t,x,y)u_{\lambda
^{\kappa}_0}(\tilde{x}) \rrho(\tilde{x})\,d\tilde{x}}{\int
_{B_r(x)}u_{\lambda^{\kappa}_0}(\tilde{x}) \rrho(\tilde
{x})\,d\tilde{x}} \\
&\leq&\zeta(x)\frac{\int_{B_r(x)}p^{\kappa}(t+1,\tilde
{x},y)u_{\lambda^{\kappa}_0}(\tilde{x}) \rrho(\tilde{x})\,d\tilde
{x}}{\int_{B_r(x)}u_{\lambda^{\kappa}_0}(\tilde{x}) \rrho(\tilde
{x})\,d\tilde{x}} \\
&\leq&\zeta(x) \frac{e^{-(t+1)\lambda^{\kappa}_0}u_{\lambda
^{\kappa}_0}(y)}{\int_{B_r(x)}u_{\lambda^{\kappa}_0}(\tilde{x})
\rrho(\tilde{x})\,d\tilde{x}}.
\end{eqnarray*}
For fixed $x$, $p^{\kappa}(t,x,y)e^{t\lkap}$ is dominated by a
constant times $u_{\lkap}(y)$, which is in~$\mL^{1}$.
The dominated convergence theorem, together with \eqref{E:heatasympt},
implies that there is a constant $c$ such that for any Borel set $U$,
%
%e4.8 #&#
\begin{eqnarray}\label{laststeptoqsd}
\qquad \lim_{t\rightarrow\infty}e^{\lambda^{\kappa}_0 t}\mathbb
{P}_{x} (X_t \in U , \tau_{\partial} > t  ) &=&\lim
_{t\rightarrow\infty} \int_0^{\infty} e^{\lambda^{\kappa}_0
t}p^{\kappa}(t,x,y)\mathbf{1}_U(y)
\rrho(y)\,dy\nonumber\\[-8pt]\\[-8pt]
&=& c u_{\lambda^{\kappa}_0}(x) \int_Uu_{\lambda^{\kappa}_0}(y)
\rrho(y) \,dy.\nonumber
\end{eqnarray}
Taking quotients,
\begin{eqnarray*}
\lim_{t\rightarrow\infty}\mathbb{P}^x ( X_t \in U|\tau
_{\partial} > t ) &=& \lim_{t\rightarrow\infty}\frac{\mathbb
{P}^x ( X_t \in U, \tau_{\partial} > t )}{\mathbb
{P}^x (\tau_{\partial} > t )} \\
&=& \frac{\lim_{t\rightarrow\infty}e^{\lambda^{\kappa}_0t}\mathbb
{P}^x (X_t \in U, \tau_{\partial} > t )}{\lim
_{t\rightarrow\infty}e^{\lambda^{\kappa}_0t}\mathbb{P}^x
(\tau_{\kappa} > t )} \\
&=& \frac{c\int_Uu_{\lambda^{\kappa}_0}(y) \rrho(y)\,dy}{c\int
_0^{\infty}u_{\lambda^{\kappa}_0}(y) \rrho(y)\,dy}.
\end{eqnarray*}
\upqed
\end{pf}\eject

For the second part of the proof we apply an argument used in \citet
{CL90} to derive properties of eigenfunctions of Schr\"{o}dinger
operators. Some modification is required to deal with the complication
that we have a domain with boundary, and we do not know a priori that
the eigenfunctions are bounded. The one-dimensional setting helps us to
overcome these complications.

%le4.5 #&#
\begin{lemma}\label{integrability}
Assume that $\lambda^{\kappa}_0<K := \liminf_{x \rightarrow\infty
}\kappa(x)$. Then the square integrable nonnegative eigenfunction
$u_{\lambda^{\kappa}_0}$ associated to the isolated
eigenvalue~$\lambda^{\kappa}_0$ is integrable with respect to the measure $\Rho$.
\end{lemma}

\begin{pf}
By \eqref{E:L2ae} and the Feynman--Kac formula
%
%e4.6 ###
%e4.9 #&#
\begin{equation}\label{invariance}
e^{-\lambda^{\kappa}_0 t} u_{\lambda_0}(x) = \mathbb{E}_x
\bigl[e^{-\int_0^t\kappa(X_s^{*})\,ds}u_{\lambda^{\kappa
}_0}(X_t^{*}),T_0 > t \bigr],
\end{equation}
for every $x \in[0,\infty)$, where $X^{*}_{s}$ is the diffusion which
is killed only at the boundary. For $t \geq0$ we define the martingale
\[
M_t = e^{-\int_0^t(\kappa-\lambda^{\kappa}_0)(X^*_s)\,ds}u_{\lambda
^{\kappa}_0}(X^*_t)\mathbf{1}_{\lbrace T_0 > t\rbrace}.
\]
By the assumption $\lambda^{\kappa}_0 < K$ there exist positive real
numbers $a$ and $\varepsilon$ such that $\kappa(x) - \lambda^{\kappa
}_0 > \varepsilon$ for every $x \in[a, \infty)$. Let $T_a$ be the
first hitting time of the set $[0,a]$.

By the optional sampling theorem we get for every $T>0$ and $x > a$
%
%e4.10 #&#
\begin{eqnarray}\label{optionalsampling}
u_{\lambda^{\kappa}_0}(x) &=& \mathbb{E}_x \bigl[ e^{-\int_0^{T_a
\wedge T}(\kappa-\lambda^{\kappa}_0)(X^*_s)\,ds}u_{\lambda^{\kappa
}_0}(X^*_{T_a \wedge T})\mathbf{1}_{\lbrace T_0 > T_a \wedge T\rbrace
} \bigr] \nonumber\\
&=& \mathbb{E}_x \bigl[e^{-\int_0^{T}(\kappa-\lambda^{\kappa
}_0)(X^*_s)\,ds}u_{\lambda^{\kappa}_0}(X^*_{T})\mathbf{1}_{\lbrace
T_a>T \rbrace} \bigr] \nonumber\\[-8pt]\\[-8pt]
&&{} + \mathbb{E}_x \bigl[e^{-\int_0^{T_a}(\kappa-\lambda
^{\kappa}_0)(X^*_s)\,ds}u_{\lambda^{\kappa}_0}(a)\mathbf
{1}_{\lbrace T_a \leq T \rbrace} \bigr] \nonumber \\
&\le& e^{-\varepsilon T} \mathbb{E}_x \bigl[u_{\lambda^{\kappa
}_0}(X^*_T)\mathbf{1}_{\lbrace T_0 > T\rbrace} \bigr]+ u_{\lambda
^{\kappa}_0}(a) \mathbb{E}_x [e^{-\varepsilon T_{a}\wedge
T} ] .\nonumber
\end{eqnarray}

By Lemma \ref{L:sqrtcompare} and the spectral theorem \eqref
{spectraltheorem3} the first term is bounded by
%
%e4.11 ###
%e4.10 ###
%e4.9 ###
%e4.8 ###
%e4.7 ###
%e4.11 #&#
\begin{eqnarray}\label{boundednesspotential}
&&e^{-\varepsilon T} \bigl(C_{\alpha}(x)\bigl\|\sqrt{L}e^{-TL}u_{\lambda
^{\kappa}_0}\bigr\|+C'_{\alpha}(x)\|e^{-TL}u_{\lambda^{\kappa}_0}\|
\bigr)\nonumber \\
&&\qquad = e^{-\varepsilon T} \biggl[C_{\alpha}(x) \biggl(\int_{0}^{\infty}
\lambda e^{-2T\lambda} d\|E^{0}u_{\lambda^{\kappa}_0}\|^2(\lambda
) \biggr)^{1/2}\nonumber \\
&&\qquad \quad \hspace*{33pt}{} +C'_{\alpha} \biggl(\int_{0}^{\infty} e^{-2T\lambda}
\,d\|E^{0}u_{\lambda^{\kappa}_0}\|^2(\lambda) \biggr)^{1/2} \biggr]\\
&&\qquad \le e^{-\varepsilon T}2T^{-1/2} \|u_{\lambda^{\kappa}_0}\|\nonumber \\
&&\qquad \onrarrow{T\to\infty} 0.\nonumber
\end{eqnarray}

We have then, from \eqref{optionalsampling} and the dominated
convergence theorem, that
%
%e4.12 #&#
\begin{eqnarray}\label{fortyone}
0 &\leq& u_{\lkap}(x) \le\lim_{T\to\infty} u_{\lambda^{\kappa
}_0}(a) \mathbb{E}_x [e^{-\varepsilon T_{a}\wedge T} ] \nonumber\\[-8pt]\\[-8pt]
& =& u_{\lambda^{\kappa}_0}(a) \mathbb{E}_x [e^{-\varepsilon
T_{a}} ].\nonumber
\end{eqnarray}
We now appeal to a basic fact from potential theory [stated and proved
in much greater generality as Proposition D.15 of \citet{DvC00}; see
also page 285 of \citet{BG68}]: There is a~constant $C(a)$ such that
for all $x\ge a$,
%
%e4.12 ###
%e4.13 #&#
\begin{equation} \label{E:potential}
\mathbb{E}_x [e^{-\varepsilon T_a};T_a < \infty ] = C(a)
g^{\varepsilon}(x,a),
\end{equation}
where $g^{\varepsilon}$ is the $\varepsilon$-potential, defined by
\[
g^{\varepsilon}(x,y) = \int_0^{\infty}e^{-\varepsilon t}p(t,x,y)\,dt,
\]
where $p(t,x,y)=p^0(t,x,y)$ denotes the integral kernel of the operator
$e^{-tL}$. Since $e^{-tL}$ is self-adjoint, the integral kernel
$p(t,x,y)$ is symmetric with respect to $\Rho$, so that from \eqref{fortyone}
\begin{eqnarray*}
\int_a^{\infty} u_{\lkap}(x) \rrho(x)\,dx&\le& u_{\lambda^{\kappa
}_0}(a)\int_{a}^{\infty}\mathbb{E}_x [e^{-\varepsilon T_a}
] \rrho(x) \,dx\\
&\le& C(a) u_{\lambda^{\kappa}_0}(a) \izf g^{\varepsilon}(x,a) \rrho
(x)\,dx \\
&=& C(a)u_{\lambda^{\kappa}_0}(a) \izf\izf e^{-\varepsilon t}
p(t,x,a) \rrho(x) \,dx \,dt \\
&=& C(a) u_{\lambda^{\kappa}_0}(a) \izf\izf e^{-\varepsilon t}
p(t,a,x) \rrho(x) \,dx \,dt \\
&\le& C(a) \varepsilon^{-1}u_{\lambda^{\kappa}_0}(a) .
\end{eqnarray*}
Since $u_{\lkap}(x) \rrho(x)$ is bounded on $[0,a]$, this completes
the proof.
\end{pf}

%re4.6 #&#
\begin{remark}
The above result reflects a general principle, which seems to be well
known to analysts and mathematical physicists: The decay of the
eigenfunctions associated with isolated eigenvalues is dictated by the
decay of Green's function, at least in regions where the potential
$\kappa$ is negligible.
\end{remark}

%s4.3 ###
%s4.3 #&#
\subsection{\texorpdfstring{Low killing at $\infty$: The recurrent case}{Low killing at infinity: The recurrent case}} \label{sec:lowkill}
We assume for the remainder of this section that $K:=\lim_{x
\rightarrow\infty}\kappa(x)$ exists. Whereas the total surviving
mass in the case $K > \lambda^{\kappa}_0$ decays at the strictly
exponential rate $e^{-\lambda_{0}^{\kappa}t}$, in the case $\lim_{x
\rightarrow\infty}\kappa(x) < \lambda_0^{\kappa}$ one typically has
%
%e4.13 ###
%e4.14 #&#
\begin{equation}\label{nonpositiverecurrence}
\lim_{t \rightarrow\infty}e^{\lambda_0^{\kappa}t} \mathbb
{P}_x(X_t \in A, \tau_{\partial} > t) = 0
\end{equation}
for every bounded Borel set $A \subset[0,\infty)$. (This can be seen
for a Brownian motion with constant\vadjust{\goodbreak} drift by direct computation.)
Equation \eqref{nonpositiverecurrence} remains true for every
diffusion, if the bottom of the spectrum of the diffusion generator is
not an eigenvalue in the $\mL^2$-sense. Thus we cannot rely upon
arguments that assume a spectral gap.

It may seem surprising that, despite the complicated relationship
between the unkilled motion and killing for determining the lifetime of
the process (and hence, whether it returns to its starting point), the
conventional transience/recurrence dichotomy for the \textit{unkilled}
process is exactly the criterion that distinguishes between convergence
and escape to infinity. We begin in this section by assuming that the
unkilled process is recurrent, which is equivalent to assuming that
$\int_0^{\infty}\rrho(x)^{-1}\,dx = \infty$, and show that this
implies convergence to quasistationarity. In particular the lowest
eigenfunction $\varphi(\lambda_0^{\kappa},\cdot)$ is integrable
(but now not necessarily square integrable) with respect to $\Rho$. In
Section \ref{sec:transient} we then address the case when the unkilled
process is transient.\vspace*{-3pt}

%th4.7 #&#
\begin{theorem}\label{qsddrift}
Let infinity be a natural boundary. Suppose that \mbox{$K< \lambda^{\kappa
}_0$}, and $\int_0^{\infty}\rrho(x)^{-1}\,dx = \infty$. Then $X_t$
started from an arbitrary compactly supported initial distribution $\nu
$ converges to the quasistationary distribution with $\Rho$-density
proportional to $\varphi(\lambda_0^{\kappa},\cdot)$. Moreover, the
asymptotic mortality rate~$\eta_{\nu}$ is independent of~$\nu$ and
equals $\lambda_0^{\kappa}$.\vspace*{-3pt}
\end{theorem}

\begin{pf}
If $X_t$ escapes to infinity then we know from Lemma \ref{SEcomb} that
\[
a(\nu,r) = \lim_{t\rightarrow\infty}\frac{\mathbb{P}_{\nu} (
\tau_{\partial}>t+r )}{\mathbb{P}_{\nu}(\tau_{\partial}>t)}=
e^{-K r}.
\]
Since by assumption $\lambda_0^{\kappa}>K$, when $\alpha>0$ part
\hyperlink{it:lambda0}{(vii)} of Lemma \ref{spectrum} tells us that $\lambda
_0>0$. The strict positivity of $\lambda_0$ together with the
assumption $\int_0^{\infty}\rrho(x)^{-1}\,dx = \infty$ allow us to
apply part \hyperlink{it:speedfinite}{(ii)} of Lemma \ref{spectrum}, to
conclude that the speed measure $\Rho$ is finite. When $\alpha=0$ and
$\Rho$ is infinite the same reasoning holds, leading to a
contradiction. Therefore we may assume, in any case, that $\Rho$ is finite.

Therefore Lemma \ref{exporder} shows that for every compactly
supported measure~$\nu$
\[
-\lim_{t\rightarrow\infty}\frac{1}{t}\log\mathbb{P}_{\nu}
(\tau_{\partial}> t ) = \lambda^{\kappa}_0.
\]
In the case of escape to infinity equations \eqref{eta} and \eqref
{elementaryfact} imply
\[
-\lim_{t\rightarrow\infty} \frac{1}{t}\log\mathbb{P}_{\nu}(\tau
_{\partial}>t) = \eta_{\nu} = K \ne\lambda_0^{\kappa}.
\]
Therefore the assumption $F(\nu,\mathbb{R}_+) = 0$ cannot be true,
and thus by Theorem \ref{Thm33} we conclude $F(\nu,\mathbb{R_+}) =
1$ and $F(\nu,\infty) = 0$. Thus $X_t$ converges from every compactly
supported initial distribution $\nu$ to the quasistationary
distribution $\varphi(\lambda_0^{\kappa},\cdot)$.\vspace*{-3pt}
\end{pf}

The above theorem has the following corollary, which in a slightly more
restrictive form already appears in the\vadjust{\goodbreak} work of \citet{CMSM95}. The
proof presented in \citet{CMSM95} suffers from a gap, so it seems to be
worth presenting an alternative (and more general) proof of the assertion.

%co4.8 #&#
\begin{Corollary}\label{CMSM}
Suppose $\kappa\equiv0$ and $\infty$ is a natural boundary point,
and the process $X_{t}$ is recurrent, with $\alpha>0$.
\begin{itemize}
\item If $\lambda_{0}>0$, then $X_t$ converges from every compactly
supported initial distribution $\nu$ to the quasistationary
distribution with $\Rho$-density proportional to $\varphi(\lambda
_0,\cdot)$.
\item If $\lambda_0 = 0$, then $X_t$ started from $\nu$ escapes to infinity.
\end{itemize}
\end{Corollary}

\begin{pf}
The first part of the assertion follows directly from Theorem~\ref
{qsddrift}. In order to prove the second assertion, observe that the function
\[
R(y):=
\frac{1}{1+\alpha}+\frac{2\alpha}{1+\alpha}\int_{0}^{y}\rrho(x)^{-1}\,dx
\]
satisfies $LR=0$; since $R(0)=1/(1+\alpha)$ and $\frac
{1}{2}R'(0)=\alpha/(1+\alpha)$ the function~$R$ coincides with the
unique eigenfunction $\varphi(0,\cdot)$. We have
\begin{eqnarray*}
\int_0^{\infty}\varphi(0,y) \rrho(y)\,dy &=& \frac{1}{1+\alpha}\int
_0^{\infty}\rrho(y)\,dy+\frac{2\alpha}{1+\alpha}\int_0^{\infty
}\rrho(y)\int_0^y\rrho^{-1}(x)\,dx\, dy \\
&=& \infty,
\end{eqnarray*}
by the assumption that $\infty$ is a natural boundary.
\end{pf}

%s4.4 ###
%s4.4 #&#
\subsection{Low killing at infinity: The transient case} \label{sec:transient}
%
%th4.9 #&#
\begin{theorem}\label{Escape}
Suppose that $\infty$ is a natural boundary point and that $\int
_0^{\infty}\rrho(x)^{-1}\,dx < \infty$. If $K<\lkap$, then $X_t$
escapes to infinity from every initial distribution. The rate of escape
is exponential with rate $\lkap-K$, in the sense that for all $z,x>0$,
%
%e4.14 ###
%e4.15 #&#
\begin{equation} \label{E:escaperate}
\limsup_{t\to\infty} \frac{1}{t}\log\mathbb{P}_x  ( X_t \leq
z |\tau_{\partial} > t ) = -(\lkap-K).
\end{equation}
\end{theorem}

\begin{pf}
Observe that the condition $\int_0^{\infty}\rrho(x)^{-1}\,dx <
\infty$ implies that for each $a \in(0,\infty)$ and each $x \in(a,
\infty)$ the unkilled diffusion (corresponding to the generator $L$)
started from $x$ has nonzero probability of never hitting $a$. For
$\varepsilon> 0$ we can choose $a = a_{\varepsilon} \in(0, \infty)$
such that $\kappa(x) \in(K-\varepsilon,K+\varepsilon)$ for every $x
\in[a, \infty)$. Then we have for every $x \in(a,\infty)$
%
%e4.15 ###
%e4.16 #&#
\begin{eqnarray}\label{denominator}
\mathbb{P}_x  ( \tau_{\partial}>t ) &=& \mathbb{E}_x
\bigl[e^{-\int_0^{t}\kappa(X_s)\,ds}, T_0>t \bigr] \nonumber \\
&\geq& e^{-(K+\varepsilon)t} \mathbb{P}_x (T_a > t ) \\
&\geq& e^{-(K+\varepsilon)t}\mathbb{P}_x (T_a = \infty ).\nonumber
\end{eqnarray}
Since $\mathbb{P}_x (T_a = \infty )$ is an increasing
function of $x$, we can apply the Markov property to see that there is
a nonzero increasing function $C(x)$ such that for all $x>0$
%
%e4.16 ###
%e4.17 #&#
\begin{equation} \label{den2}
\mathbb{P}_x  ( \tau_{\partial}>t ) \ge\mathbb{P}_x
 ( X_1\ge a+1 ) \cdot\inf_{x'\ge a+1} \mathbb{P}_{x'}
 ( \tau_{\partial}>t-1 )\ge
C(x) e^{-(K+\varepsilon)t}.\hspace*{-35pt}
\end{equation}
Note that $\inf_{x'\ge a+1} \mathbb{P}_{x'}  ( \tau_{\partial
}>t-1 ) > 0$ because there is no explosion.
On the other hand, for any fixed $z\ge0$ we can apply the bound \eqref
{E:sqrtcompare2} and Lemma~\ref{L:supportbase} to see that
%
%e4.18 #&#
\begin{eqnarray} \label{E:num}
\mathbb{P}_x ( X_t\leq z, \tau_{\partial} > t ) &=&
\bigl(e^{-tL^{\kappa}}\mathbf{1}_{[0,z]}\bigr) (x) \nonumber\\[-8pt]\\[-8pt]
&\leq& \bigl(C_{\alpha}(x)\lkap+C'_{\alpha} \bigr)\bigl\|\indic
_{[0,z]}\bigr\| e^{-t\lkap}\nonumber
\end{eqnarray}
for all $t>1/2\lkap$. Combining \eqref{den2} and \eqref{E:num}, we
see that there is a constant~$C'$ such that
%
%e4.17 ###
%e4.19 #&#
\begin{equation} \label{E:numden}
\mathbb{P}_x ( X_t\leq z   | \tau_{\partial} > t )
\le
\bigl\|\indic_{[0,z]}\bigr\| \frac{C'}{C(x)} e^{-(\lkap-K-\varepsilon)t}.
\end{equation}

We conclude that for all $x\ge a_{\varepsilon}$,
\[
\limsup_{t\to\infty} \frac{1}{t}\log\mathbb{P}_x  ( X_t \leq
z \cond\tau_{\partial} > t ) \le-(\lkap-K)+\varepsilon.
\]
By Proposition \ref{P:Harnackconstantlimit}, since $\varepsilon$ is
arbitrary, we conclude that the limsup is no more than $-(\lkap-K)$.

In particular, we have shown that the process escapes to infinity. By
Lemma~\ref{SEcomb}, it follows that $\lim_{t\to\infty} \mathbb
{P}_{x}  \{\tp>t+1   |   \tp>t \}=e^{-K},$ from
which we conclude using \eqref{elementaryfact} that
\[
\lim_{t\to\infty} t^{-1}\log\mathbb{P}_{x}  \{\tp>t  \}=-K.
\]
Lemmas \ref{L:lowspec} and \ref{L:supportbase} tell us that
\[
\lim_{t\to\infty} t^{-1}\log\mathbb{P}_{x}  \{X_t \leq z
\}\ge-\lkap,
\]
from which we conclude that
\[
\liminf_{t\to\infty} \frac{1}{t}\log\mathbb{P}_x  ( X_t \leq
z \cond\tau_{\partial} > t ) \ge-(\lkap-K),
\]
completing the proof of \eqref{E:escaperate}.
\end{pf}

If $\kappa$ is eventually constant---that is, for some $a$ we have
$\kappa(x)=K$ for all \mbox{$x\ge a$}---then we can strengthen the conclusion
of Theorem \ref{Escape} slightly.

%co4.10 #&#
\begin{Corollary}\label{compareMM}
Suppose that $\kappa$ is eventually constant and that \mbox{$\lkap> 0$}.
Then for every $x , z\in(0,\infty)$
\[
\sup_{t} e^{(\lkap-K) t} \mathbb{P}_x (X_t \leq z |\tau
_{\partial} > t ) < \infty.
\]
\end{Corollary}

\begin{pf}
If $\kappa$ is eventually constant, then \eqref{denominator} and
\eqref{den2} hold with $\varepsilon=0$, hence \eqref{E:numden} as
well.\vadjust{\goodbreak}
\end{pf}

%re4.11 #&#
\begin{remark}
The case $\kappa\equiv0$ corresponds to the setting considered in
\citet{MSM01}. Theorem 4 of \citet{MSM01} includes a slightly
weaker version of the result in Corollary \ref{compareMM}, obtained by
different methods. The above theorem shows that when $\kappa\equiv0$
the principal eigenvalue $\lambda_0^{\kappa}$ gives the exponential
convergence rate at which~$X_t$ escapes to infinity.
\end{remark}

As already mentioned in Remark \ref{qsl}, a quasilimiting distribution
$\tilde{\nu}$, which in our case is a probability measure on
$(0,\infty)$ is always quasistationary in the sense that for every
Borel set $A \subset(0,\infty)$,
\[
\mathbb{P}_{\tilde{\nu}} (X_t \in A |\tau_{\partial}>
t ) = \tilde{\nu}(A);
\]
but the converse need not hold true. In the cases where we know there
is no quasilimiting distribution, though, because the process escapes
to $\infty$, we can show that there is also no quasistationary distribution.

%co4.12 #&#
\begin{Corollary}
Let $\infty$ be a natural boundary, with $\lambda_0^{\kappa}>K$
and\break
$\int_0^{\infty}\rrho(x)^{-1}\,dx < \infty$.
Then there is no quasistationary distribution.
\end{Corollary}

\begin{pf}
Assume that $\tilde{\nu}$ is a general quasistationary distribution.
The measure $\tilde{\nu}$ is absolutely continuous with respect to
$\Rho$ with a positive continuous density $g\dvtx[0,\infty)\rightarrow
(0,\infty)$ (for a sketch of the proof of this fact we refer to the
\hyperref[app]{Appendix}). There is a $\lambda$ such that $\mathbb{P}_{\tilde{\nu
}}(\tau_{\partial}>t) = e^{-\lambda t}$. By \eqref{den2}, for any
positive $\varepsilon$,
%
%e4.18 ###
%e4.20 #&#
\begin{equation} \label{E:killatmost}
e^{-\lambda t}=\P_{\tilde\nu}\{\tp>t\}\ge e^{-(K+\varepsilon
)t}\int C(x)\,d\tilde\nu(x),
\end{equation}
which means that $\lambda\le K$. For any fixed $x_0>0$,
%
%e4.19 ###
%e4.21 #&#
\begin{eqnarray} \label{E:killatleast}
\qquad \P_{\tilde\nu}\{X_t\le z,  \tp>t \}&=&  \bigl\langle
g,e^{-tL^{\kappa}}\mathbf{1}_{(0,z]} \bigr\rangle\nonumber\\
&=& \bigl\langle g\indic_{[0,x_0]},e^{-tL^{\kappa}}\mathbf
{1}_{(0,z]} \bigr\rangle+ \bigl\langle r \indic_{(x_0,\infty
)},e^{-tL^{\kappa}}\mathbf{1}_{(0,z]} \bigr\rangle\\
&\le& \bigl\|g\indic_{[0,x_0]} \bigr\|\cdot\bigl \|e^{-tL^{\kappa
}}\mathbf{1}_{(0,z]}  \bigr\|+\sup_{x\ge x_{0}} \P_{x} \{X_{t}\le
z,  \tp>t\}.\nonumber
\end{eqnarray}
Since $\|g\indic_{[0,x_{0}]}\|$ and $\|\indic_{(0,z]}\|$ are both
finite, we can use \eqref{spectraltheorem3} and \eqref{E:num} to see
that there is a constant $B$ such that
%
%e4.20 ###
%e4.22 #&#
\begin{equation} \label{E:killatleast2}
\P_{\tilde{\nu}}\{X_t\le z,  \tp>t \}\le Be^{-t\lkap}.
\end{equation}
Combining \eqref{E:killatmost} and \eqref{E:killatleast2}, we see
that for all positive $t$,
%
%e4.21 ###
%e4.23 #&#
\begin{equation} \label{E:stable}
\tilde{\nu} ([0,z] )=\P_{\tilde{\nu}}\{X_t\le z  |
  \tp>t \}\le Be^{-(\lkap-K)t},
\end{equation}
so $\tilde{\nu}$ must be identically 0 on $[0,\infty)$.
\end{pf}

%s4.5 ###
%s4.5 #&#
\subsection{Processes that may not hit 0} \label{sec:htrans}
Consider a process which is killed only at 0 (i.e., with $\kappa\equiv
0$). If the process is not almost surely absorbed at 0
eventually---that is, if $\mathbb{P}_x(T_0 = \infty) > 0$---we may
wish to condition the process at time $t$ on being killed eventually,
but not yet. That is, we consider the long-time asymptotics of
\[
\mathbb{P}_x\bigl(X_t \in\cdot\,| T_0 \in(t,\infty)\bigr).\vspace*{-2pt}
\]
Conditions of this kind can often be found in the analogous problems in
the theory of branching processes. This problem can be reduced to our
previous analysis by an h-transform. The function $h(x) = \mathbb
{P}_x(T_0 < \infty)$ is harmonic, and by general theory [see \citet
{rP95}, Chapter 4, Sections 3 and 10] the process $(X_t)$ conditioned
to hit $0$ corresponds to the generator $L^h$ whose action is given by
\[
L^{h}f=\biggl(\frac{1}{h}L(h f)\biggr)(x) = -\frac{1}{2}f''(x)+
\biggl(-b(x)-\frac{h'(x)}{h(x)} \biggr)f'(x).\vspace*{-2pt}
\]
The process associated to the operator $L^h$ can again be defined by
Dirichlet form techniques, and the associated family of measures on the
path space is denoted by $\tilde{\mathbb{P}}_x$. As explained above
we have
\[
\mathbb{P}_x(\cdot\,| T_0 < \infty) = \tilde{\mathbb{P}}_x(\cdot).\vspace*{-2pt}
\]
The operator $L^h$ can be realized as a self-adjoint operator on the
Hilbert space $\mL^2((0,\infty), h(x)^{2}\rrho(x)\,dx)$. The
transformation $V\dvtx\mL^2((0,\infty),\break h(x)^{2}\rrho(x)\,dx)\to\mL
^2((0,\infty),\rrho(x)\,dx)$ defined by $Vf=fh$ is unitary, and defines
a unitary equivalence between $L$ and $L^{h}$, so the spectrum is
invariant under $h$-transforms. In particular, positivity of the bottom
of the spectrum of $L$ implies the positivity of the spectrum of $L^h$.
Since absorption is certain with respect to the measure $\tilde
{\mathbb{P}}_x$ we can apply our previous results in order to conclude
that for every Borel set $A \subset(0,\infty)$
\[
\lim_{t \rightarrow\infty}\mathbb{P}_x\bigl(X_t \in A | T_0 \in
(t,\infty)\bigr) = \frac{\int_A\tilde{\varphi}^h(\lambda_0,x)h(x)\rrho
(x)\,dx} {\int_0^{\infty}\tilde{\varphi}^h(\lambda_0,x)h(x)\rrho(x)\,dx},\vspace*{-2pt}
\]
where $\tilde{\varphi}^h(\lambda_0,x)$ is the unique solution of
$(L^h-\lambda_0)u = 0$, which satisfies $\tilde{\varphi}^h(\lambda
_0,0)=0$ and $(\tilde{\varphi}^h)'(\lambda_0,0)=1$.\vspace*{-3pt}

%s4.6 ###
%s4.6 #&#
\subsection{\texorpdfstring{The case of an entrance boundary at $\infty$}{The case of an entrance boundary at infinity}} \label{sec:entrancebound}
$\!\!\!$Observe that $\int_0^{\infty} \rrho(x)^{-1}\,dx = \infty$ if $\infty$
is an entrance boundary. This follows from the fact that in this
situation the total speed measure $\int_0^{\infty}\rrho(x) \,dx$ must
be finite. Thus, the situation is essentially the same as in Theorem
\ref{qsddrift}. Indeed, we always have convergence to
quasistationarity if $\infty$ is an entrance boundary.\vspace*{-3pt}

%th4.13 #&#
\begin{theorem}\label{nichtnatuerlichqsd}
Assume that $0$ is regular and that $\infty$ is an entrance boundary.
Then the bottom of the spectrum is an isolated eigenvalue with
associated nonnegative eigenfunction $u_{\lambda_0^{\kappa}}$. From
every compactly supported initial distribution $\nu$, the process
$X_t$ converges to the distribution with density $u_{\lambda_0^{\kappa
}}/\izf u_{\lkap}(x)\rrho(x)\,dx$ with respect to $\Rho$.\vadjust{\goodbreak}
\end{theorem}

\begin{pf}
The first assertion follows from Theorem \ref{nichtnatuerlich}. Lemma
\ref{L:L1L2} directly implies that $X_t$ converges to the
quasistationary distribution $u_{\lambda_0^{\kappa}}$ from every
compactly supported initial distribution if and only if $\int
_0^{\infty}u_{\lambda_0^{\kappa}}(y) \rrho(y)\,dy$ is finite. Since
we are assuming that $0$ is regular and $\infty$ is an entrance
boundary the speed measure $\Rho$ must be finite. Thus the $\mL
^{2}(\Rho)$ function $u_{\lkap}$ is also in $\mL^{1}(\Rho)$.\vspace*{-2pt}
\end{pf}

%s4.7 ###
%s4.7 #&#
\subsection{\texorpdfstring{Existence and uniqueness of quasistationary distributions when \mbox{$\kappa\equiv0$}}
{Existence and uniqueness of quasistationary distributions when kappa equivalent 0}}
In this short section we first reformulate the criterium for existence
of quasistationary distributions in the case $\kappa\equiv0$ and
$\mathbb{P}_x(T_0< \infty)=1$. This allows a direct comparison with
the criterium for the uniqueness of the quasistationary distriubution,
which has been recently established in \citet{6authors}. We consider
only the case $\alpha>0$, since otherwise there is no killing at all,
and this is merely a classical situation of a stationary distribution.
The interesting point in the next result consists of the fact that the
existence of some exponential moment of the first hitting time $T_0$ of
$0$ is equivalent to the existence of quasistationary distributions for
any $\alpha>0$.\vspace*{-2pt}
%
%th4.14 #&#
\begin{theorem}
Let $0$ be regular and let infinity be inaccessible. Moreover, suppose
that $\alpha\in(0,\infty]$, $\kappa\equiv0$ and $\mathbb{P}_x(T_0
< \infty) = 1$.
\begin{itemize}[(ii)]
\item[(i)] There exists a quasistationary distribution if and only if
for some $\varepsilon> 0$ and some (hence every) $x >0$
\[
\mathbb{E}_x [ e^{\varepsilon T_0} ] < \infty.
\]
\item[(ii)] There exists a unique quasistationary distribution if and
only if for every $a > 0$ there exists $y_a > 0$ such that
\[
\sup_{x > y_a}\mathbb{E}_x [e^{a T_{y_a}} ] < \infty.
\]
This is true if and only if infinity is an entrance boundary.\vspace*{-2pt}
\end{itemize}
\end{theorem}

\begin{pf}
Assertion (ii) follows from assertion (i) in combination with Theorem~7.3 of \citet{6authors}.
In order to prove assertion (i) let us first
assume that there exists a quasistationary distribution $\nu$. Then
there exists $\bar{\lambda}$ such that $\mathbb{P}_{\nu} (\tau
_{\partial}>t ) = e^{-\bar{\lambda}t}$. Since hitting $0$ is
certain we conclude that $\bar{\lambda}>0$. As shown in Lemma \ref{lemA} of
the \hyperref[app]{Appendix} the measure $\nu$ is absolutely continuous with respect
to $\Gamma$ with a strictly positive and continuous density $\varphi
$. Therefore, we get for $0<a<b$ and some positive constant $c>0$
%
%e4.22 ###
%e4.24 #&#
\begin{eqnarray}
e^{-\bar{\lambda}t}&=&\mathbb{P}_{\nu} (\tau_{\partial}>t) =
\int_0^{\infty}\varphi(y)\mathbb{P}_y(\tau_{\partial}>t)\Gamma
(dy)\nonumber \\
&\geq& c\int_a^b\varphi(y)\bigl(e^{-tL^{0,\alpha}}\mathbf
{1}_{(a,b)}\varphi\bigr)(y)\Gamma(dy) \\
&=& c \int_{[0,\infty)}e^{-t\lambda}\,d\bigl\|E^{0,\alpha}_{\lambda
}\bigl(\mathbf{1}_{(a,b)}\varphi\bigr)\bigr\|^2.\nonumber
\end{eqnarray}
Using Lemma \ref{L:supportbase} we therefore have $\lambda
_0^{0,\alpha}=\inf\operatorname{supp} d\|E^{0,\alpha}(\mathbf
{1}_{(a,b)}\varphi)\|^2(\lambda) \geq\bar{\lambda}>0$. Since the
essential spectra of $L^{0,\alpha}$ and $L^{0,\infty}$ coincide,
either $\lambda_0^{0,0}$ is strictly positive or $0$ is an isolated
eigenvalue. According to Lemma \ref{spectrum}\hyperlink{it:noisolated}{(vi)} the latter case
cannot occur. But this means, according to Corollary \ref{CMSM}, that
$\varphi(\lambda_0^{0,0},y)\Gamma(dy)/\int_0^{\infty}\varphi
(\lambda_0^{0,0},y)\Gamma(dy)$ is the quasilimiting distribution of
the diffusion killed at $0$ and therefore $\lim_{t \rightarrow\infty
}\frac{1}{t}\log\mathbb{P}_x(T_0>t)=-\lambda_0^{0,0}$. Hence
$\mathbb{E}_x[e^{\varepsilon T_0}]<\infty$ for every $0<\varepsilon<
\lambda_0^{0,0}$.

Assume now that $\mathbb{E}_x[e^{\varepsilon T_0} ]<\infty$ for
some $0<\varepsilon$. Then obviously
\[
\lim_{t \rightarrow\infty}\frac{1}{t}\log\mathbb{P}_x(T_0>t)<0.
\]
Using implication \eqref{elementaryfact} we see that the asymptotic
killing rate $\eta_x$ is strictly bigger than $0$. By Lemma \ref
{SEcomb} the process $X_t$ with absorption at $0$ does not escape to
infinity; hence it converges. By Corollary \ref{CMSM} we then have
\mbox{$\lambda_0^{0,\infty}>0$}. Using the same argument as in the first
part of the proof of assertion (i) we conclude that $\lambda
_0^{0,\alpha}>0$, which by Corollary \ref{CMSM} implies the existence
of a~quasistationary distribution.
\end{pf}

Thus uniqueness of quasistationary distributions is equivalent to the
``time of implosion from infinity into the interior'' having
exponential moments of all orders, whereas existence of a
quasistationary distribution is equivalent to the existence of \textit
{some} exponential moment of the first hitting time of $0$. Both
results together account for the existence and uniqueness of
quasistationary distributions.

%re4.15 #&#
\begin{remark}
It seems to be a rather general principle that there are three
possibilities. The first possibility is the nonexistence of
quasistationary distributions. If there exists a quasistationary
distribution, then it is either unique or there is a whole continuum of
quasistationary distributions parameterized by a real interval. This is
at least true for birth and death processes on the nonnegative
integers; cf. \citet{jC78}.
\end{remark}

%s5 ###
%s5 #&#
\section{The dichotomy and the integrability of the principal eigenfunction}
According to the basic dichotomy of \citet{quasistat}, as extended
here, we know that under the assumptions $K \ne\lambda_0^{\kappa}$
and the nonaccessibility of infinity either $X_t$ converges to the
quasistationary distribution $\varphi(\lambda_0^{\kappa},\cdot)/
\int_0^{\infty}\varphi(\lambda_0^{\kappa},y)\,d\gamma(y)$ or $X_t$
escapes to infinity. Moreover, we have shown that escape to infinity
occurs (under these assumptions) if and only if the boundary point
infinity is natural, the underlying unkilled diffusion is transient and
$\lambda_0^{\kappa}>K$.

Another way of expressing this dichotomy is in terms of the
integrability of the principal eigenfunction. It follows without much
effort from Theorem \ref{generallocalMandl} [see, e.g., Proposition
2.3 in \citet{quasistat}], that $\int_0^{\infty} \varphi(\lambda
_0^{\kappa},y) \Gamma(dy) = \infty$ implies escape to infinity. Is
integrability of the principal eigenfunction actually equivalent to
Yaglom convergence, at least under the condition $K \ne\lambda
_0^{\kappa}$? We answer in the affirmative, stating the result as
a~theorem because of its salience, although it might strictly be seen as
a~fairly direct corollary to the results of Section~\ref{sec:qstat}.

%th5.1 #&#
\begin{theorem}\label{thm:integrability}
Assume that $0$ is regular and that infinity is not accessible.
Moreover, suppose that $K \ne\lambda_0^{\kappa}$. Then $X_t$ escapes
to infinity if and only if
\[
\int_0^{\infty} \varphi(\lambda_0^{\kappa},y)\,d\Gamma(y) =
\infty.
\]
If $\int_0^{\infty} \varphi(\lambda_0^{\kappa},y)\,d\Gamma(y) <
\infty$, then $X_t$ converges to the quasilimiting distribution $\frac
{\varphi(\lambda_0^{\kappa},y)\,d\Gamma(y)}{\int_0^{\infty} \varphi
(\lambda_0^{\kappa},y)\,d\Gamma(y)}$.
\end{theorem}
\begin{pf}
All we need to show is that
\[
\int_0^{\infty} \varphi(\lambda_0^{\kappa},y)\,d\Gamma(y) = \infty
\]
holds when there is escape to infinity; that is, in the case $\lambda
_0^{\kappa}>K$, $\infty$ is natural, and $\int_0^{\infty}\gamma
(y)^{-1}\,dy < \infty$. In all other cases we know that $X_t$
converges to quasistationary, and in particular the principal
eigenfunction is integrable.

Under these assumptions we know from the proof of Theorem \ref{Escape}
[see equation \eqref{den2}] that for every $\varepsilon> 0$ there
exists a nontrivial, nonnegative increasing function $C_{\varepsilon
}(\cdot)$ such that for $y > 0$ and $t > 0$
%
%e5.1 ###
%e5.1 #&#
\begin{equation}\label{1rev}
\mathbb{P}_y (\tau_{\partial}>t ) \geq C_{\varepsilon
}(y) e^{-(K+\varepsilon)t}.
\end{equation}
Let us assume that $m:=\int_0^{\infty}\varphi(\lambda_0^{\kappa
},y)\,d\Gamma(y)<\infty$ and show that this gives a~contradiction.
Integration of \eqref{1rev} with respect to the probability measure
$m^{-1}\varphi(\lambda_0^{\kappa},y)\,d\Gamma(y)$ gives
%
%e5.2 ###
%e5.2 #&#
\begin{eqnarray}\label{Arev}
&&m^{-1}\int_0^{\infty}\varphi(\lambda_0^{\kappa},y)\mathbb
{P}_y (\tau_{\partial}>t )\,d\Gamma(y) \nonumber\\[-8pt]\\[-8pt]
&&\qquad \geq
e^{-(K+\varepsilon)t}m^{-1}\int_0^{\infty}C_{\varepsilon}(y)\varphi
(\lambda_0^{\kappa},y)\,d\Gamma(y).\nonumber
\end{eqnarray}
On the other hand, using the symmetry of the semigroup we have
%
%e5.3 ###
%e5.3 #&#
\begin{eqnarray}\label{L1norm}
&&m^{-1}\int_0^{\infty}\varphi(\lambda_0^{\kappa},y)\mathbb
{P}_y (\tau_{\partial}>t )\,d\Gamma(y) \nonumber\\[-8pt]\\[-8pt]
&&\qquad = m^{-1}\int
_0^{\infty} (e^{-tL^{\kappa}}\varphi(\lambda_0^{\kappa},\cdot
) )(y)\,d\Gamma(y).\nonumber
\end{eqnarray}
Now observe that $\varphi(\lambda_0^{\kappa},\cdot)$ is $\lambda
_0^{\kappa}$-subinvariant [see, e.g., Lemma 7.7 in \citet{quasistat}],
that is,
%
%e5.4 ###
%e5.4 #&#
\begin{equation}\label{subinv}
e^{-tL^{\kappa}}\varphi(\lambda_0^{\kappa},\cdot) \leq e^{-\lambda
_0^{\kappa}t}\varphi(\lambda_0^{\kappa},\cdot).\vadjust{\goodbreak}
\end{equation}
Using \eqref{L1norm}, \eqref{Arev} and \eqref{subinv}, for all $t>0$,
%
%e5.5 ###
%e5.5 #&#
\begin{eqnarray}
m&\geq&\int_0^{\infty}\varphi(\lambda_0^{\kappa},y)\mathbb
{P}_y (\tau_{\partial}>t )\,d\Gamma(y) \nonumber\\[-8pt]\\[-8pt]
&\geq& e^{(\lkap-K-\varepsilon)t}\int_0^{\infty}C_{\varepsilon}(y)\varphi(\lambda
_0^{\kappa},y)\,d\Gamma(y).\nonumber
\end{eqnarray}
As we know that $\lkap-K-\varepsilon> 0$ for $\varepsilon$
sufficiently small, and the integral is assumed nonzero, the right-hand
side goes to $\infty$ as $t\to\infty$, which is a~contradiction if
$m$ is finite. Therefore $m=\int_0^{\infty}\varphi(\lambda
_0^{\kappa},y)\,d\Gamma(y)= \infty$.
\end{pf}

%apA #&#
\begin{appendix}

\section*{Appendix}\label{app}
\setcounter{equation}{0}

In this \hyperref[app]{Appendix} we sketch a proof of the regularity of quasistationary
distributions of one-dimensional diffusions with one regular boundary.

\renewcommand{\thelemmaA}{A}
%le1 #&#
\begin{lemmaA}\label{lemA}
Let $L^{\kappa}$ be one of the self-adjoint realizations considered in
this work, of the Sturm--Liouville expression $\tau+ \kappa$ in $\mL
^2$; and let $\tilde{\nu}$ be a~quasistationary distribution. Then
$\tilde{\nu}$ is absolutely continuous with respect to the measure
$\Gamma$, with a positive and continuous density $g\dvtx[0,\infty
)\rightarrow\mathbb{R}$.
\end{lemmaA}

\begin{pf}
The main assertion of the lemma will be almost obvious to readers who
are familiar with regularity theory for stationary distributions.
Observe that the main point here is the continuity up to the boundary
$0$. Indeed, the main strategy we follow is very similar to the case of
stationary distributions. Straightforward arguments show that $\tilde
{\nu}$ is absolutely continuous with respect to the measure $\Gamma$.
Denote by $g$ the density of $\tilde{\nu}$ with respect to $\Gamma$.
The equation
\[
e^{-\lambda t}\tilde{\nu}(f) = \mathbb{E}_{\tilde{\nu}}
[f(X_t);\tau_{\partial}>t ];\qquad  \lambda\geq0, f \in
C_c((0,\infty)),
\]
which results from quasistationarity of $\tilde{\nu}$, implies that
%
%e5.6 ###
%eA.1 #&#
\begin{equation}\label{E:ellipticequation}
 \forall f \in C^{\infty}_c((0,\infty))\dvtx\langle\tilde{\nu},
(L^{\kappa}+\lambda)f\rangle= \int g(x)(L^{\kappa}+\lambda)f(x)\,
d\Gamma(x) = 0.\hspace*{-35pt}
\end{equation}
This means that for any $0<c<d$,
\[
g \in\mathcal{D}(T_{c,d}^*) \quad \mbox{and}\quad  T_{c,d}^*g = 0,
\]
where $T_{c,d}^*$ denotes the adjoint [taken in the Hilbert space $\mL
^2((c,d),\Gamma)$] of the minimal operator $T_{c,d}$ defined as the
restriction of the differential operator $L^{\kappa}+ \lambda$ to
$C^{\infty}_c((c,d))$. The domain of $T_{c,d}^*$ is given by
%
%eA.2 #&#
\begin{eqnarray}
\qquad \mathcal{D}(T_{c,d}^*) &=&  \biggl\lbrace f \in\mL^2((c,d),\Gamma)
| f,\gamma f' \mbox{ absolutely continuous in $(c,d)$
and}\nonumber\\[-8pt]\\[-8pt]
&&\hspace*{107pt}\frac{-1}{2\gamma}(\gamma f')' + (\kappa+\lambda) f \in
\mL^2((c,d),\Gamma) \biggr\rbrace,\nonumber
\end{eqnarray}
and for $f \in\mathcal{D}(T_{c,d}^*)$ one has $T_{c,d}^* = \frac
{-1}{2\gamma}(\gamma f')' + (\kappa+ \lambda)f$. Since $c,d \in
(0,\infty)$ are arbitrary we conclude that
%
%e5.7 ###
%eA.3 #&#
\begin{equation}
\frac{-1}{2\gamma}(\gamma f)' (x) + (\kappa+ \lambda)f (x)=0
\qquad \mbox{in $(0,\infty)$}.
\end{equation}
Due to the regularity of the boundary point $0$ we conclude (using
standard ODE theory) that
\[
\lim_{x \rightarrow0+}g(x) \in\mathbb{R}
\]
exists.
\end{pf}

%s5.1 ###
%sA.1 #&#
\subsection{Very short summary of extension theory}
In this section we give a~summary of the analytic results we applied in
Lemma \ref{L:SL}, a full account of which can be found in Chapter 10
of \citet{jW00}. For an arbitrary symmetric operator $S$ in a complex
Hilbert space $\mathcal{H}$ let $\operatorname{Ran}(S-z)$ denote the
image of the linear operator $S-z$ and let $\operatorname
{Ran}(S-z)^{\perp}$ denote its orthonal complement. Set $\beta
_+:=\dim\operatorname{Ran}(S+i)^{\perp}$ and $\beta_-:=\dim
\operatorname{Ran}(S-i)^{\perp}$.
Observe that in the symmetric case $ \dim\operatorname{Ran}(S\pm
i)^{\perp} = \dim\operatorname{Ker}(S^* \mp i)$. Thus the deficiency
indices give the dimension of the solution space of the equation $(S^*
\mp i)u=0$ ($u \in\mathcal{H}$).
The pair $(\beta_+,\beta_-)$ are called the deficiency numbers of $S$
and describes the ``number'' of self-adjoint extensions of $S$. If in
the notation of the beginning of Section 2.1 $S=\tau_{p,q,V}$ is, for
example, a (minimal) Sturm--Liouville differential expression on the
interval $(a_1,a_2)$, then one always has $\beta_+=\beta_-$. Moreover,
\[
(\beta_+,\beta_-)=
\cases{
(0,0), &\quad if $a_1$ and $a_2$ are both limit point boundaries, \vspace*{2pt}\cr
(1,1), &\quad if one boundary point is limit point and \cr
&\quad  the other boundary point limit circle,\vspace*{2pt}\cr
(2,2), &\quad if $a_1$ and $a_2$ are both limit circle boundaries.
}
\]
We note that these formulas follow immediately from the definition of
limit-point type/limit-circle type, the remark appearing after
Definition \ref{lplc} and the fact that the deficiency indices give
the dimensions of the space of solutions to eigenvalue equations. The
case $(\beta_+,\beta_-)=(0,0)$ corresponds to essential
self-adjointness, that is, the case, where there is is only one
self-adjoint extension. Moreover, if $(\beta_+,\beta_-)=(m,m)$ a
symmetric extension $T$ of $S$ [i.e., $\mathcal{D}(S)\subset\mathcal
{D}(T), T\restriction\mathcal{D}(S)=S$] is self adjoint if and only
if $\mathcal{D}(T)/\mathcal{D}(S)$ has dimension $m$. Thus in this
case self-adjoint extensions of $S$ are exactly the $m$-dimensional
symmetric extensions.
\end{appendix}

\section*{Acknowledgments}
The authors would like to thank Ross Pinsky for sending them his work
[\citet{rP09}] prior to publication and for his useful remarks. Moreover
we would like to thank Gady Kozma for sending us his work concerning
the Davies conjecture. The hospitality of Worcester College, Oxford for
M. Kolb in support of this collaboration is gratefully acknowledged.
Moreover, the authors thank the referee for a very thorough and careful
reading of the paper. %Some support for both authors came from the New
%Dynamics of Ageing program.

%suskaldyti doi

% imsref loaded by dianan, 2011-04-08 15:18:36
%

\printaddresses

\end{document}